\documentclass[11pt]{amsart}

\newif\ifdraft\draftfalse

\usepackage{amssymb, enumerate, xspace, graphicx, url}
\usepackage[latin1]{inputenc}
\usepackage[mathscr]{eucal}
\usepackage[all]{xy}
\SelectTips{cm}{}

\ifdraft\usepackage[notcite, notref]{showkeys}\fi

\numberwithin{subsection}{section}

\allowdisplaybreaks[1]

\newenvironment{enumeratea}
{\begin{enumerate}[\upshape (a)]}
{\end{enumerate}}

\newenvironment{enumeratei}
{\begin{enumerate}[\upshape (i)]}
{\end{enumerate}}

\newtheorem*{namedtheorem}{\theoremname}
\newcommand{\theoremname}{testing}

\newcommand\Fields{\operatorname{Fields}}

\newcommand\Sets{\operatorname{Sets}}

\newcommand\chr{\operatorname{char}}
\newcommand\Ind{\operatorname{Ind}}

\def\et{{\rm \acute e t}}
\def\Et{{\rm \acute E t}}
\def\Sch{{\rm Sch}}
\newcommand\fet{{\rm f\acute et}}

\newcommand\Gal{\operatorname{{\rm Gal}}}

\newcommand\Orb{\operatorname{\bf Orb}}
\newcommand\Spec{\operatorname{Spec}}

\newcommand\ram{\mathrm{ram}}  
\newcommand\unr{\mathrm{unr}}
\newcommand\Kram{K_{\ram}}
\newcommand\Kunr{K_{\unr}}

\newtheorem{theorem}{Theorem}[section]
\newtheorem{thm}{Theorem}[section]
\newtheorem{proposition}[theorem]{Proposition}
\newtheorem{proposition-definition}[theorem]{Proposition-Definition}
\newtheorem{corollary}[theorem]{Corollary}
\newtheorem{lemma}[theorem]{Lemma}

\newtheorem{prop}[thm]{Proposition}

\newtheorem{lem}[thm]{Lemma}

\theoremstyle{definition}
\newtheorem{definition}[theorem]{Definition}

\newtheorem{example}[theorem]{Example}

\newtheorem{remark}[theorem]{Remark}
\newtheorem{question}[theorem]{Question}

\theoremstyle{remark}

\numberwithin{equation}{section}

 \newcommand\cB{\mathcal{B}}
 \newcommand\cD{\mathcal{D}}
 
\newcommand\cG{\mathcal{G}} \newcommand\cH{\mathcal{H}}

\newcommand\cM{\mathcal{M}} \newcommand\cN{\mathcal{N}}
\newcommand\cO{\mathcal{O}}

\newcommand\cU{\mathcal{U}} 
\newcommand\cW{\mathcal{W}} \newcommand\cX{\mathcal{X}}
\newcommand\cY{\mathcal{Y}} 

\renewcommand\AA{\mathbb{A}} 
\newcommand\CC{\mathbb{C}} 
 
\newcommand\GG{\mathbb{G}}

 \newcommand\PP{\mathbb{P}}

 \newcommand\ZZ{\mathbb{Z}}

 \newcommand\bH{\mathbf{H}}

\newcommand\bM{\mathbf{M}} 
 \newcommand\bP{\mathbf{P}}

 \newcommand\bX{\mathbf{X}}
\newcommand\bY{\mathbf{Y}}

\newcommand\rma{\mathrm{a}}

\newcommand\rmm{\mathrm{m}}

 \newcommand\rmt{\mathrm{t}}
 
\newcommand\rmw{\mathrm{w}}

\newcommand\fM{\mathfrak{M}}

\newcommand\arr{\ifinner \to\else\longrightarrow\fi}

\newcommand\arrto{\ifinner\mapsto\else\longmapsto\fi}
\newcommand\larr{\longrightarrow}

\newcommand{\xarr}{\xrightarrow}

\renewcommand\H{\operatorname{H}}

\newcommand\eqdef{\overset{\mathrm{\scriptscriptstyle def}} =}

\def\displaytimes_#1{\mathrel{\mathop{\times}\limits_{#1}}}

\def\displayotimes_#1{\mathrel{\mathop{\bigotimes}\limits_{#1}}}

\newcommand\HOM{\operatorname{HOM}}

\DeclareMathOperator{\Aut}{Aut}
\DeclareMathOperator{\spc}{\mathrm{sp}}

\newcommand\Br{\operatorname{Br}}

\newcommand\ind{\operatorname{ind}}

\newcommand\spec{\operatorname{Spec}}

\newcommand\codim{\operatorname{codim}}

\renewcommand\projlim{\varprojlim}

\newdir{ >}{{}*!/-5pt/@{>}}

\newcommand\double{\mathbin{\rightrightarrows}}

\newcommand\doublelong[2]{\mathbin{\xymatrix{{}\ar@<3pt>[r]^{#1}
\ar@<-3pt>[r]_{#2}&}}}

\newcommand{\underaut}
{\mathop{\underline{\mathrm{Aut}}}\nolimits}

\newlength{\ignora}

\newcommand{\ed}{\operatorname{ed}}
\newcommand{\cd}{\operatorname{cd}}

\newcommand{\dm}{Deligne--Mumford\xspace}
\newcommand{\lci}{local complete intersection\xspace}

\newcounter{steps}

\newcommand{\trdeg}{\operatorname{tr\,deg}}
\newcommand{\mmu}{\boldsymbol{\mu}}

\newcommand{\GL}{\mathrm{GL}}

\newcommand{\PGL}{\mathrm{PGL}}
\newcommand{\Hom}{\mathrm{Hom}}

\newcommand{\gm}{\GG_{\rmm}}

\newcommand{\dr}[1]{(\mspace{-3mu}(#1)\mspace{-3mu})}
\newcommand{\ds}[1]{[\mspace{-2mu}[#1]\mspace{-2mu}]}

\renewcommand{\setminus}{\smallsetminus}

\begin{document}

\title[Essential dimension of moduli of curves] {Essential dimension of 
moduli of curves\\and other algebraic stacks}

\author[Brosnan]{Patrick Brosnan$^\dagger$}

\author[Reichstein]{Zinovy Reichstein$^\dagger$}

\author[Vistoli]{Angelo~Vistoli$^\ddagger$ $\quad \quad \quad \quad \quad$
with an appendix by Najmuddin Fakhruddin}

\address[Brosnan, Reichstein]{Department of Mathematics\\
The University of British Columbia\\
1984 Mathematics Road\\
Vancouver, B.C., Canada V6T 1Z2}

\address[Vistoli]{Scuola Normale Superiore\\Piazza dei Cavalieri 7\\
56126 Pisa\\Italy}

\address[Fakhruddin]{School of Mathematics, 
Tata Institue of Fundamental Research, 
Homi Bhabha Road, Mumbai 400005, India}
\email[Brosnan]{brosnan@math.ubc.ca}
\email[Reichstein]{reichst@math.ubc.ca}
\email[Vistoli]{angelo.vistoli@sns.it}
\email[Fakhruddin]{naf@math.tifr.res.in}

\begin{abstract} 
  In this paper we consider questions of the following type. 
  Let $k$ be a base field and $K/k$ be a field extension. Given 
  a geometric object $X$ over a field $K$ (e.g. a smooth curve of genus $g$)
  what is the least transcendence degree of a field of definition of
  $X$ over the base field $k$?  In other words, how many independent
  parameters are
  needed to define $X$? To study these questions we introduce a notion
  of essential dimension for an algebraic stack. Using the resulting theory,
  we give a complete answer to the question above when the geometric
  objects $X$ are smooth, stable or hyperelliptic curves. The appendix, 
  written by Najmuddin Fakhruddin, answers this question 
  in the case of abelian varieties.
\end{abstract}
\thanks{$^\dagger$Supported in part by an NSERC discovery grant}
\thanks{$^\ddagger$Supported in part by the PRIN Project ``Geometria
sulle variet\`a algebriche'', financed by MIUR}

\maketitle

\setcounter{tocdepth}{1}
\tableofcontents

\section{Introduction}
\label{s.intro}

This paper was motivated by the following question. 

\begin{question} \label{q.curves} Let $k$ be a field and $g \ge 0$ 
be an integer.  What is the smallest integer $d$ such that for every
field $K/k$, every smooth curve $X$ of genus $g$ defined over K descends 
to a subfield $k \subset K_0 \subset K$ with $\trdeg_k K_0 \le d$?
\end{question}

Here by ``X descends to $K_0$" we mean that the exists a curve $X_0$ over $K_0$ such that $X$ is $K$-isomorphic to $X_0 \times_{\Spec K_0} \Spec K$.

In order to address this and related questions, we will introduce 
and study the notion of essential dimension for algebraic stacks;
see~\S\ref{s.definition}.  The essential 
dimension $\ed \cX$ of a scheme $\cX$ is simply the dimension 
of $\cX$; on the other 
hand, the essential dimension of the classifying stack $\cB_{k}G$
of an algebraic group $G$ is the essential dimension of $G$ 
in the usual sense; see~\cite{reichstein} or~\cite{bf1}.
The notion of essential dimension of a stack is meant to bridge these
two examples.  The minimal integer $d$ in 
Question~\ref{q.curves} is the essential dimension of the
moduli stack of smooth curves $\cM_g$. We show that
$\ed \cX$ is finite for a broad class of algebraic stacks 
of finite type over a field; see Corollary~\ref{cor.finiteness}.
This class includes all \dm stacks and all quotient stacks 
of the form $\cX = [X/G]$, where $G$ a linear algebraic group.

Our main result is the following theorem.  

\begin{theorem}
\label{thm.curves}
Let $\cM_{g,n}$ (respectively, $\overline{\cM}_{g,n}$) 
be the stacks of $n$-pointed smooth (respectively, stable) algebraic 
curves of genus $g$ over a field $k$ of characteristic $0$. Then
\[ \ed\cM_{g,n} = 
   \begin{cases} 
   2         & \text{if }(g,n)=(0,0)\text{ or } (1,1);\\
   0         & \text{if }(g,n)=(0,1)\text{ or } (0,2);\\
   +\infty   & \text{if }(g,n)=(1,0);\\
   5         & \text{if }(g,n)=(2,0);\\
   3g-3 + n  & \text{otherwise}.
\end{cases}
\]

Moreover for $2g-2+n > 0$ we have $\ed\overline{\cM}_{g,n} = \ed \cM_{g,n}$.
\end{theorem}

In particular, the values of $\ed\cM_{g, 0} = \ed \cM_g$ give 
a complete answer to Question~\ref{q.curves}. 

Note that $3g-3+n$ is the dimension of the moduli space $\bM_{g,n}$
in the stable range $2g-2+n > 0$ (and the dimension of the stack in all
cases); the dimension of the moduli space represents an obvious lower
bound for the essential dimension of a stack. The first
four cases are precisely the ones where a generic object in
$\cM_{g,n}$ has non-trivial automorphisms, and $(g,n) =
(1,0)$ is the only case where the automorphism group scheme 
of an object of $\cM_{g,n}$ is not affine.

Our proof of Theorem~\ref{thm.curves} for $(g, n) \ne (1, 0)$
relies on two results of independent interest. One is  
the ``Genericity Theorem"~\ref{thm:generic} which says 
that the essential dimension of a smooth integral 
\dm stack satisfying an
appropriate separation hypothesis is the sum of its dimension and the
essential dimension of its generic gerbe. This somewhat surprising
result implies that the essential dimension of a non-empty open substack
equals the essential dimension of the stack. In particular, it proves
Theorem~\ref{thm.curves} in the cases where a general curve in
$\cM_{g,n}$ has no non-trivial automorphisms. It also brings 
into relief the important role played by gerbes in this theory.

The second main ingredient in our proof of Theorem~\ref{thm.curves} 
is the following formula, which we use to compute 
the essential dimension of the generic gerbe.

\begin{theorem}
\label{t.edGerbe}
Let $\cX$ be a gerbe over a field $K$ banded by a group $G$.
Let $[\cX] \in \H^2(K, G)$ be the Brauer class of $\cX$.

\begin{enumeratea}

\item If $G = \GG_m$ and $\ind [\cX]$ is a prime power 
then $\ed \cX = \ind {[\cX]} - 1$.

\item If $G = \mmu_{p^r}$, where $p$ is a prime and $r \ge 1$, 
then $\ed \cX = \ind {[\cX]}$.

\end{enumeratea}
\end{theorem}

Our proof of this theorem can be found in 
the preprint~\cite[Section 7]{brv1}.
A similar argument was used by N. Karpenko and A. Merkurjev 
in the proof of~\cite[Theorem 3.1]{km2}, which generalizes 
Theorem~\ref{t.edGerbe}(b). For the sake of completeness,
we include an alternative proof of Theorem~\ref{t.edGerbe}
in \S\ref{sect.gerbe}.

Theorem~\ref{t.edGerbe} has a number of applications
beyond Theorem~\ref{thm.curves}. Some of these have already 
appeared in print.  In particular, we used Theorem~\ref{t.edGerbe}
to study the essential dimension of spinor groups in~\cite{brv3}, 
N. Karpenko and A. Merkurjev~\cite{km2} used it 
to study the essential dimension of finite $p$-groups, 
and A. Dhillon and N. Lemire~\cite{dhillon-lemire}
used it, in combination with the Genericity 
Theorem~\ref{thm:generic}, to give an upper bound for the essential 
dimension of the moduli stack of $\operatorname{SL}_n$-bundles 
over a projective curve. In this paper
Theorem~\ref{t.edGerbe} (in combination with Theorems~\ref{thm:generic}) 
is also used to study the essential dimension of 
the stacks of hyperelliptic curves (Theorem~\ref{t.hyperelliptic}) 
and, in the appendix written by Najmuddin Fakhruddin, of principally 
polarized abelian varieties.

In the case where $(g, n) = (1, 0)$ Theorem~\ref{thm.curves}
requires a separate argument, which is carried out 
in \S\ref{s.Tate}. In this case Theorem~\ref{thm.curves}
is a consequence of the fact that the group schemes of $l^{n}$-torsion
points on a Tate curve has essential dimension $l^n$, where $l$ 
is a prime.

\subsection*{Acknowledgments}
We would like to thank the Banff International Research Station in
Banff, Alberta (BIRS) for providing the inspiring meeting place where
this work was started.  We are grateful to  J.~Alper, K.~Behrend, 
C.-L. Chai, D.~Edidin, 
A. Merkurjev, B.~Noohi, G.~Pappas, 
M. Reid and B.~Totaro for helpful conversations.  

\section{The essential dimension of a stack}
\label{s.definition}

Let $k$ be a field. We will write $\Fields_k$ for the category of
field extensions $K/k$.   
Let $F\colon\Fields_k \arr\Sets$ be
a covariant functor.     

\begin{definition} 
  \label{def.merkurjev} Let $a\in F(L)$, where $L$ is an object of
  $\Fields_k$.  We say that $a$ \emph{descends} to an intermediate field
  $k \subset K\subset L$ or that $K$ is a \emph{field
  of definition} for $a$ if $a$ is in the image of the induced map
  $F(K) \arr F(L)$.

  The \emph{essential dimension} $\ed a$ of $a \in F(L)$ is the
  minimum of the transcendence degrees $\trdeg_{k}K$ taken over all
  intermediate fields $k \subseteq K \subseteq L$ such 
  that $a$ descends to $K$.

  The essential dimension $\ed F$ of the functor $F$ is the supremum
  of $\ed a$ taken over all $a\in F(L)$ with $L$ in $\Fields_k$.
  We will write $\ed F=-\infty$ if $F$ is the empty functor.

  These notions are relative to the base field $k$. To emphasize this,
  we will sometimes write $\ed_k a$ or $\ed_k F$ instead of $\ed a$
  or $\ed F$, respectively.
\end{definition}         


The following definition singles out a class of functors that is 
sufficiently broad to include most interesting examples, 
yet ``geometric" enough to allow one to get a handle on their 
essential dimension. 

\begin{definition} \label{def.ed-stack} Suppose $\cX$ is an algebraic
stack over $k$.  The \emph{essential dimension} $\ed \cX$ of $\cX$
is defined to be the essential dimension
of the functor $F_{\cX}\colon\Fields_k \arr\Sets$ which
sends a field $L/k$ to the set of isomorphism classes of objects in
the groupoid $\cX(L)$.\footnote{In the literature the functor
$F_{\cX}$ is sometimes denoted by $\widehat{\cX}$ or $\overline{\cX}$.}

As in Definition~\ref{def.merkurjev}, we will
write $\ed_k \cX$ when we need
to be specific about the dependence on the base field $k$.
Similarly for $\ed_k \xi$, where
$\xi$ is an object of $F_{\cX}$.
\end{definition}

\begin{example} \label{ex.edG}
Let $G$ be an algebraic group defined over $k$ and
$\cX = \cB_{k}G$ be the classifying stack of $G$. 
Then $F_{\cX}$ is the Galois cohomology functor
sending $K$ to the set $\H^1(K, G)$ of isomorphism classes 
of $G$-torsors over $\Spec(K)$, in the fppf
topology.  The essential 
dimension of this functor is a numerical invariant of $G$, 
which, roughly speaking, measures the complexity of $G$-torsors 
over fields.  This number is usually denoted by $\ed_k G$ or (if
$k$ is fixed throughout) simply by $\ed G$; following this
convention, we will often write $\ed G$ in place of $\ed \cB_{k}G$.
Essential dimension was originally introduced and 
has since been extensively studied
in this context; see e.g.,~\cite{bur, reichstein, ry, kordonskii0,
ledet, jly, bf1, lemire, cs, garibaldi}. The more general
Definition~\ref{def.merkurjev} is due to A.~Merkurjev; see
\cite[Proposition 1.17]{bf1}.
\end{example}          

\begin{example} \label{ex.ed-variety}
Let $\cX = X$ be a scheme of finite type over a field $k$,
and let $F_X \colon\Fields_k  \arr\Sets$ denote the functor
given by $K\arrto X(K)$.  Then an easy argument due to Merkurjev
shows that $\ed F_X=\dim X$; see~\cite[Proposition 1.17]{bf1}.

In fact, this equality remains true for any algebraic space $X$.
Indeed, an algebraic space $X$ has a stratification by schemes $X_i$.
Any $K$-point $\eta\colon\Spec K \arr X$ must land in one 
of the $X_i$.  Thus $\ed X=\max\ed X_i=\dim X$.
\qed
\end{example}

\begin{example} \label{ex.ed-curves}
Let $\cX = \cM_{g, n}$ be the stack of smooth algebraic curves 
of genus $g$. Then the functor $F_{\cX}$ sends $K$
to the set of isomorphism classes of $n$-pointed smooth algebraic
curves of genus $g$ over $K$. Question~\ref{q.curves} asks about the 
essential dimension of this functor in the case where $n = 0$.
\end{example}

\begin{example} \label{ex.forms}
Suppose a linear algebraic group $G$ is acting on an algebraic space
$X$ over a field $k$. We shall write $[X/G]$ for the quotient stack
$[X/G]$. Recall that $K$-points of $[X/G]$ are by definition
diagrams of the form
\begin{equation} \label{e.functor}
 \xymatrix{
T \ar@{->}[r]^{\psi} \ar@{->}[d]^{\pi} &  X \cr
\Spec(K) &  }
\end{equation}
where $\pi$ is a $G$-torsor and $\psi$ is a $G$-equivariant map.
The functor $F_{[X/G]}$ associates with a field $K/k$ the set
of isomorphism classes of such diagrams.

In the case where $G$ is a special group (recall that this means 
that every $G$-torsor over $\Spec(K)$ is split, for every field $K/k$)
the essential dimension of $F_{[X/G]}$ has been previously studied in
connection with the so-called ``functor of orbits"
$\Orb_{X, G}$ given by the formula
\[ \text{$\Orb_{X, G}(K) \eqdef $ set of $G(K)$-orbits in $X(K)$.} \]
Indeed, if $G$ is special, the functors $F_{[X/G]}$ and $\Orb_{X, G}$
are isomorphic; an isomorphism between them is given by
sending an object~\eqref{e.functor} of $F_{[X/G]}$
to the $G(K)$-orbit of the point $\psi s \colon \Spec(K)  \arr X$, where
$s \colon \Spec(K) \arr T$ is a section of $\pi \colon T \to \Spec(K)$.

Of particular interest are the natural $\GL_n$-actions 
on $\AA^N$ = affine space 
of homogeneous polynomials of degree $d$ in $n$ variables and
on $\PP^{N-1}$ = projective space of degree $d$ hypersurfaces 
in $\PP^{n-1}$, where $N = {n + d - 1 \choose d}$ is the number 
of degree $d$ monomials in $n$ variables. For general $n$ and $d$
the essential dimension of the functor of orbits in these cases 
is not known. Partial results can be found~\cite{bf2} 
and~\cite[Sections 14-15]{ber}. Additional results in
this setting will be featured in a forthcoming paper.
\end{example}

%

\begin{remark} \label{rem.finite}
If the functor $F$ in Definition~\ref{def.merkurjev}
is limit-preserving, a condition satisfied in all 
cases of interest to us, then every element $a \in F(L)$
descends to a field $K \subset L$ that is finitely generated over $k$.
Thus in this case $\ed a$ is finite.
In particular, if $\cX$ is an algebraic stack over $k$,
$\ed \xi$ is finite for every object $\xi \in \cX(K)$ 
and every field extension $K/k$; the limit-preserving property 
in this case is proved in~\cite[Proposition 4.18]{LMB},

In~\S\ref{sect.fiber-dimension} we will show that, 
in fact, $\ed \cX < \infty$ for a broad class of algebraic stacks $\cX$;
cf.~Corollary~\ref{cor.finiteness}. On the other hand,
there are interesting examples where 
$\ed \cX = \infty$; see~Theorem~\ref{thm.curves} 
or~\cite{bs}.
\end{remark}

The following observation is a variant of~\cite[Proposition 1.5]{bf1}.

\begin{proposition}
\label{p.extensions}
Let $\cX$ be an algebraic stack over $k$, and let $K$ 
be a field extension of $k$. Then $\ed_{K}\cX_{K} \leq \ed_{k}\cX$.
\end{proposition}

Here, as in what follows, we denote by $\cX_{K}$ the stack $\spec K \times_{\spec k}\cX$.

\begin{proof}
If $L/K$ is a field extension,
then the natural morphism $\cX_K(L) \arr \cX(L)$ is an equivalence.
Suppose that $M/k$ is a field of definition for an object $\xi$ in
$\cX (L)$. Let $N$ be a composite of $M$ and $K$ over $k$. Then $N$ is 
a field of definition for $\xi$, $\trdeg_K N\leq \trdeg_k M$, and
the proposition follows.
\end{proof}

\section{A fiber dimension theorem}
\label{sect.fiber-dimension}
 
We now recall Definitions (3.9) and (3.10) from~\cite{LMB}.
A morphism $f\colon\cX \arr\cY$ of algebraic stacks (over $k$) is said to be
{\em representable} if, for every $k$-morphism $T \to \cY$, where $T$ is 
an affine $k$-scheme, the fiber product $\cX\times_{\cY} T$ is 
representable by a an algebraic space over $T$. A representable morphism
$f\colon\cX \arr\cY$ is said to be \emph{locally of finite type 
and of fiber dimension $\le d$} if the projection
$\cX\times_{\cY} T \to T$ is also locally of finite type 
over $T$ and every fiber has dimension $\leq d$. 

\begin{example} \label{ex.rel-dim}
Let $G$ be an algebraic group defined over $k$, and let 
$X \to Y$ be a $G$-equivariant morphism of $k$-algebraic spaces,
locally of finite type and of relative dimension $\le d$.  Then 
the induced map of quotient stacks 
$[X/G] \to [Y/G]$ is representable, locally of finite type
and of relative dimension $\le d$.
\end{example} 

The following result may be viewed as a partial 
generalization of the fiber dimension theorem 
(see~\cite[Exercise II.3.22 or Proposition III.9.5]{hartshorne}) 
to the setting where schemes are replaced by stacks and 
dimension by essential dimension.

\begin{theorem}
\label{thm.fiber-dimension}
Let $d$ be an integer, $f \colon \cX \arr
\cY$ be a representable $k$-morphism of algebraic stacks
which is locally of finite type and of fiber dimension at most $d$. 
Let $L/k$ be a field, $\xi \in \cX(L)$. Then

\begin{enumeratea}

\item $\ed_k \xi \leq \ed_k f(\xi) + d$, and

\item $\ed_k \cX \leq \ed_k \cY + d$.

\end{enumeratea}
 In particular, if 
$\ed_k \cY$ is finite, then so is  $\ed_k \cX$.
\end{theorem}

\begin{proof} (a) By the definition of $\ed_k f(\xi)$ we can 
find an intermediate field $k \subset K \subset L$ and a morphism
  $\eta \colon\Spec K \arr \cY$ such that 
  $\trdeg_k K \leq \ed f(\xi)$ and the following diagram commutes.
\begin{equation*}
\xymatrix{
\Spec L    \ar[r]^-{\xi}\ar[d]  & \cX\ar[d]^{f}\\
\Spec K   \ar[r]^-{\eta}        & \cY      \\
}
\end{equation*}
Let $\cX_K \eqdef \cX\times_{\cY}\Spec K$.  By the hypothesis, $\cX_K$ is an
algebraic space, locally of finite type over $K$ and of relative
dimension at most $d$.  By the commutativity of the diagram above, the
morphism $\xi \colon\Spec L \arr \cX$ factors through $\cX_K$:  
\begin{equation*}
\xymatrix{
\Spec L   \ar@/^/[rrd]^-{\xi} \ar[rd]^-{\xi_0} \ar@/_/[rdd] &  & \\  
 & \cX_K \ar[r] \ar[d] & \cX\ar[d]^{f}\\
 & \Spec K   \ar[r]^-{\eta}        & \cY      \\ }
\end{equation*}
Moreover, $\xi$ factors through $K(p)$, where 
$p$ denotes the image of $\xi_0$ in $\cX_K$.  
Since $\cX_K$ has dimension at most
$d$ over $K$, we have $\trdeg_K K(p)\leq d$.  Therefore, 
\[ \trdeg_k K(p) = \trdeg_k K + \trdeg_K K(p) \leq \ed f(\xi) + d \]
and part~(a) follows.

\smallskip
Part~(b) follows from (a) by taking the maximum 
on both sides over all $L/k$ and all $\xi \in \cX(L)$. 
\end{proof}

\begin{corollary} \label{cor.lower-bound1}
Consider an action of an algebraic 
group $G$ on an algebraic space $X$, defined over 
a field $k$. Assume $X$ 
is locally of finite type over $k$. Then
\[ \ed_k G \ge \ed_k [X/G] - \dim X \, . \]
\end{corollary}

\begin{proof} The natural $G$-equivariant map 
$X \to \spec k$ gives rise to
a map $[X/G] \to \cB_{k}G$ of quotient stacks.
This latter map is locally of finite type and of relative 
dimension $\leq \dim X$; see Example~\ref{ex.rel-dim}.
Applying Theorem~\ref{thm.fiber-dimension}(b)
to this map, we obtain the desired inequality.
\end{proof}

\begin{corollary}\label{cor.finiteness} {\rm (Finiteness of
essential dimension)}
  Let $\cX$ be an algebraic stack of finite type over $k$.
  Suppose that for any algebraically closed extension $\Omega$
  of $k$ and any
  object $\xi$ of $\cX(\Omega)$ the group scheme
  $\underaut_{\Omega}(\xi) \arr \spec\Omega$ is affine. Then
  $\ed_{k}\cX < \infty$.
\end{corollary}

Note that Corollary~\ref{cor.finiteness} fails without
the assumption that all the $\underaut_{\Omega}(\xi)$ are affine.
For example, by Theorem~\ref{thm.curves}, $\ed\cM_{1,0}=+\infty$.

\begin{proof}
We may assume without loss of generality that
$\cX = [X/G]$ is a quotient stack for some affine
algebraic group $G$ acting on an algebraic space $X$. Indeed,
by a Theorem of Kresch~\cite[Proposition 3.5.9]{kresch} $\cX$
is covered by quotient stacks $[X_i/G_i]$ of this form
and hence, $\ed\cX = \max_i \, \ed {[X_i/G_i]}$.

If $\cX = [X/G]$ then
by Corollary~\ref{cor.lower-bound1}, 
\[ \ed {[X/G]} \le \ed_k G  + \dim(X)  \, . \]
The desired conclusion now follows from the well-known
fact that $\ed_k G < \infty$ for any affine algebraic
group $G$; see~\cite[Theorem 3.4]{reichstein}
or~\cite[Proposition 4.11]{bf1}.
\end{proof}

\section{The essential dimension of a gerbe over a field}
\label{sect.gerbe}

The goal of this section is to prove Theorem~\ref{t.edGerbe}
stated in the Introduction.  We proceed by briefly 
recalling some background material on gerbes 
from \cite[p.~144]{milne} and \cite[IV.3.1.1]{G}, 
and on canonical dimension from~\cite{km} and~\cite{ber}.

{\bf Gerbes.} Let $\cX$ be a gerbe defined over a field $K$ 
\emph{banded} by an abelian $K$-group scheme $G$.
In particular, $X$ is a stack over $K$ which becomes isomorphic 
to $\cB_K G$ over the algebraic closure of $K$.

There is a notion of equivalence of gerbes banded by $G$; the
set of equivalence classes is in a natural bijective correspondence with
the group $\H^{2}(K, G)$.  The identity element of 
$\H^{2}(K, G)$ corresponds to the class of the neutral gerbe $\cB_{K}G$.
Recall that the group $\H^2(K,\GG_m)$ is canonically isomorphic 
to the Brauer group $\Br K$ of Brauer equivalence classes 
of central simple algebras over $K$. Here, as usual, $\GG_m$ denotes the
multiplicative group scheme over $K$. 

{\bf Canonical dimension.}
Let $X$ be a smooth projective variety defined over a field $K$.
We say that $L/K$ is a {\em splitting field} for $X$ if
$X(L) \ne \emptyset$. A splitting field $L/K$ is called \emph{generic}
if for every splitting field $L_0/K$ there exists a $K$-place $L \to L_0$.
The {\em canonical dimension} $\cd X $ of $X$ is defined as the
minimal value of $\trdeg_K(L)$, where $L/K$ ranges over
all generic splitting fields. Note that the function field
$L = K(X)$ is a generic splitting field of $X$; see~\cite[Lemma 4.1]{km}.
In particular, generic splitting fields exist and $\cd X$ is finite.
If $X$ is a smooth complete projective variety over $K$ then $\cd X$ has
the following simple geometric interpretation: $\cd X$ is
the minimal value of $\dim(Y)$, as $Y$ ranges over the closed 
$K$-subvarieties of $X$, which admit a rational map $X \dasharrow Y$ 
defined over $K$; see~\cite[Corollary 4.6]{km}.

The {\em determination functor} $D_X\colon \Fields_K\to\Sets$ is defined 
as follows.  For any field extension $L/K$, $D_X(L)$ is the empty set, 
if $X(L) = \emptyset$, and a set consisting of one element
if  $X(L) \neq \emptyset$. The natural map $D(L_1) \to D(L_2)$ 
is then uniquely determined for any $K \subset L_1 \subset L_2$.
It is shown in~\cite{km} that if $X$ is a complete regular 
$K$-variety then 
\begin{equation} \label{e.determination}
\cd X = \ed D_X \, .
\end{equation}

Of particular interest to us will be the case where 
$X$ is a Brauer--Severi variety over $K$.  Let $m$ be 
the index of $X$.  If $m = p^a$ is a prime power then 
\begin{equation} \label{e.cd-primary}
\cd X = p^a - 1 \, ; 
\end{equation}
see \cite[Example 3.10]{km} or~\cite[Theorem 11.4]{ber}. 

If $m = p_{1}^{a_{1}} \dots p_{r}^{a_{r}}$ is the prime 
decomposition of $m$ then 
the class of $X$ in $\Br L$ is the sum of
classes $\alpha_{1}$, \dots,~$\alpha_{r}$ whose indices
are $p_{1}^{a_{1}}$, \dots,~$p_{r}^{a_{r}}$.
Denote by $X_{1}$, \dots, $X_{r}$ the Brauer--Severi varieties
associated with $\alpha_{1}$, \dots, $\alpha_{r}$.
It is easy to see that
$K(X_1 \times \dots \times X_r)$ is a generic splitting field for $X$.
Hence,
\[ \cd X \le \dim(X_1 \times \dots \times X_r) =
p_1^{a_1} + \dots + p_r^{a_r} - r \, .  \]
J.-L. Colliot-Thélène, N. Karpenko and 
A. Merkurjev~\cite{ctkm} conjectured that equality holds, i.e.,
\begin{equation}
\label{e.conjecture}
\cd X = p_1^{a_1} + \dots + p_r^{a_r} - r \, .
\end{equation}
As we mentioned above, this in known
to be true if $m$ is a prime power (i.e., $r = 1$).
Colliot-Thélène, Karpenko and Merkurjev also
proved~\eqref{e.conjecture} for $m = 6$; see \cite[Theorem 1.3]{ctkm}.
Their conjecture remains open for all other $m$.

\begin{theorem}
\label{t.cdP}
Let $d$ be an integer with $d >1$. Let $K$ be a field 
and $x \in \H^2(K, \mmu_d)$. Denote the image of $x$ in $\H^2(K, \GG_m)$ by $y
$,
the $\mmu_d$-gerbe associated with $x$ by $\cX \to \Spec(K)$,
the $\GG_m$-gerbe associated with $y$ by $\cY \to \Spec(K)$, 
and the Brauer--Severi variety associated with $y$ by $P$. Then
\begin{enumeratea}

\item $\ed \cY = \cd P$ and

\item $\ed \cX = \cd P + 1$.

\end{enumeratea}
In particular, if the index of $x$ is a prime power $p^r$ then
$\ed \cY  = p^r - 1$ and $\ed \cX  = p^r$.
\end{theorem}

\begin{proof}
The last assertion follows from (a) and (b) by~\eqref{e.cd-primary}.

\smallskip
(a)  The functor $F_{\cY}\colon\Fields_K\to\Sets$ sends a field 
  $L/K$ to the empty set, if $P(L) = \emptyset$, and to a set 
  consisting of one point, if $P(L) \neq \emptyset$. In other words,
  $F_Y$ is the determination functor $D_P$ introduced above.
  The essential dimension of this functor is $\cd P $; 
 see~\eqref{e.determination}.

\smallskip (b) First note that the natural map $\cX\to\cY$ is of finite
type and representable of relative dimension $\le 1$. By
Theorem~\ref{thm.fiber-dimension}(b) we conclude that $\ed \cX \leq
\ed \cY +1$. By part (a) it remains to prove the opposite inequality,
$\ed \cX \geq \ed \cY  + 1$. We will do this by constructing an
object $\alpha$ of $\cX$ whose essential dimension is $\ge \ed \cY  + 1$.

We will view $\cX$ as a torsor for $\cB_{K}\mmu_d$ in the following
sense. There exist maps 
\begin{align*}
\cX\times\cB_{K}\mmu_d &\larr \cX \\
  \cX\times\cX     &\larr \cB_{K}\mmu_d
\end{align*}
satisfying various compatibilities, where the first map is the ``action"
of $\cB_{K}\mmu_d$ on $\cX$ and the second map is the ``difference'' of two
objects of $\cX$. For the definition and a discussion
of the properties of these maps, 
see~\cite[Chapter IV, Sections 2.3, 2.4 and 3.3]{G}.  (Note that, in
the notation of Giraud's book, $\cX\wedge\cB_{K}\mmu_d\cong \cX$ and the
action operation above arises from the map $\cX\times\cB_{K}\mmu_d\to
\cX\wedge\cB_{K}\mmu_d$ given in Chapter IV, Proposition 2.4.1.   The
difference operation, which we will not use here, arises similarly
from the fact that, in Giraud's notation, $\HOM(\cX,\cX)\cong\cB_{K}\mmu_d$.)

Let $L = K(P)$ be the function field of $P$. Since $L$ splits $P$, we 
have a natural map $a\colon\Spec L \to \cY$. Moreover since $L$ 
is a generic splitting field for $P$, 
\begin{equation} 
\label{e.cd1}
\ed a = \cd P  = \ed \cY,\\
\end{equation}
where we view $a$ as an object in $\cY$.
Non-canonically lift $a\colon\Spec L\to\cY$ to a map $\Spec L\to\cX$ (this can be 
done, because $\cX\to\cY$ is a $\gm$-torsor).  Let $\Spec L(t)\to\cB_{L}\mmu_d$ 
denote the map classified by 
$(t)\in \H^1(L(t), \mmu_d)=L(t)^{\times}/L(t)^{\times d}$.   
Composing these two maps, we obtain an object
$$
\alpha\colon\Spec L(t)\to \cX\times\cB_{L}\mmu_d\to\cX.
$$
in $\cX(L(t))$. Our goal is to prove that $\ed \alpha \ge \ed \cY + 1$.
In other words, given a diagram of the form 
\begin{equation} 
\label{e.oldclaim}
\xymatrix{
\Spec L(t)\ar[r]^{\quad \alpha} \ar[d] &\cX\\
\Spec M\ar[ru]^{\beta}              & \\
} 
\end{equation}
where $K \subset M \subset L$ is an intermediate field, we
want to prove the inequality $\trdeg_K(M) \ge \ed \cY + 1$. Assume the 
contrary:
there is a diagram as above with $\trdeg_K(M) \le   \ed \cY$.
Let $\nu \colon L(t)^* \to \ZZ$ be the usual discrete valuation 
corresponding to $t$ and consider two cases.

\smallskip
{\bf Case 1.} Suppose
the restriction $\nu_{| M}$ of $\nu$ to $M$ is non-trivial.
Let $M_0$ denote the residue field of $\nu$ and $M_{\geq 0}$ denote 
the valuation ring.  Since $\Spec M\to \cX\to\cY$, 
there exists an $M$-point of $P$.   Then by the
valuative criterion of properness for $P$, there exists an 
$M_{\geq 0}$-point and thus an $M_{0}$-point of $P$.  
Passing to residue fields, we obtain the diagram
\[ \xymatrix{
\Spec L\ar[r]^-{a}\ar[d] &\cY   \\
\Spec M_0\ar[ru]              & \\
}  
\]
which shows that 
$\ed a \le \trdeg_K M_0 = \trdeg_K M-1\leq\ed \cY -1$,
contradicting~\eqref{e.cd1}. 

\smallskip
{\bf Case 2.} Now suppose the restriction of $\nu$ to $M$ is trivial.
The map $\Spec L\to \cX$ sets up an
isomorphism $\cX_L\cong\cB_L\mmu_d$.  The map $\Spec L(t)\to\cX$ factors
through $\cX_L$ and thus induces a class in 
$\cB_{L} \mmu_d(L(t)) = \H^1(L(t),\mmu_d)$.  This class 
is $(t)$. Tensoring the diagram~\eqref{e.oldclaim}
with $L$ over $K$, we obtain
$$
\xymatrix{
\Spec L(t)\otimes L\ar[r]^{\alpha}\ar[d] & \cX_L\cong\cB_L \mmu_d\\
\Spec M\otimes L\ar[ru]^{\beta}             & \\
} 
$$ 
Recall that $L = K(P)$ is the function field of $P$. 
Since $P$ is absolutely irreducible, the tensor products
$L(t)\otimes L$ and $M \otimes L$ are fields.
The map $\Spec M\otimes L\to \cB_L\mmu_d$ is classified by some 
$m\in (M\otimes L)^{\times}/(M\otimes L)^{\times d}=\H^1(M\otimes L,\mmu_d)$.   
The image of $m$ in $L(t) \otimes L$ is equal to $t$ modulo
$d$-th powers. We will now derive a contradiction by comparing
the valuations of $m$ and $t$.

To apply the valuation to $m$, we lift $\nu$ from $L(t)$ 
to $L(t) \otimes L$. That is, we define $\nu_L$ as the valuation on
$L(t) \otimes L = (L \otimes L)(t)$ 
corresponding to $t$. Since
$\nu_L(t) = \nu(t) = 1$, we conclude that
$\nu_L (m)\equiv 1\pmod d$. This shows that
$\nu_L$ is not trivial on $M \otimes L$ and thus
$\nu$ is not trivial on $M$, contradicting our assumption. 
This contradiction completes the proof of part (b).
\end{proof}

\begin{corollary}
\label{cor.lower-bound2}
Let $ 1\arr Z \arr G \arr Q\arr 1$
denote an extension of group schemes over a field $k$ with $Z$ 
central and isomorphic to (a) $\gm$ or (b) $\mmu_{p^r}$ for 
some prime $p$ and some $r \ge 1$.
Let $\ind(G,Z)$ as the maximal value of
$\ind\bigl(\partial_K(t)\bigr)$ as
$K$ ranges over all field extensions of $k$ and $t$ ranges over all
torsors in $\H^1(K,Q)$. If $\ind(G, Z)$ is a prime power
(which is automatic in case (b)) then
\[
\ed_{k} G \geq \ind(G,Z) - \dim G.
\]
\end{corollary}

\begin{proof} Choose $t \in H^1(K, Z)$ so that 
$\ind\bigl(\partial_K(t)\bigr)$ attains its maximal value,
$\ind(G, Z)$. Let $X \to \Spec(K)$ be the $Q$-torsor 
representing $t$.
Then $G$ acts on $X$ via the projection $G \to Q$,
and $[X/G]$ is the $Z$-gerbe over $\Spec(K)$ corresponding 
to the class $\partial_K(t)\in \H^2(K,Z)$.  
By Theorem~\ref{t.edGerbe}. 
\[ \ed [X/G] = \begin{cases} 
\text{$\ind\bigl(\partial_K(t)\bigr) - 1$ in case (a),} \\
\text{$\ind\bigl(\partial_K(t)\bigr)$ in case (b).} \end{cases} \]
Since $\dim(X) = \dim(Q)$, applying 
Corollary~\ref{cor.lower-bound1} to 
the $G$-action on $X$,  
we obtain
\[ \ed_{K} G_K \geq  \begin{cases} 
\text{$(\ind(G,Z) - 1) - 
\dim Q = \ind(G, Z) - \dim(G)$ in case (a),}  \\ 
\text{$\ind(G,Z) - \dim Q = \ind(G, Z) - \dim(G)$ in case (b).}  
\end{cases} \]
Since $\ed_k G \ge \ed_K G_K$ (see~\cite[Proposition 1.5]{bf1} or
our Proposition~\ref{p.extensions}), the corollary follows.
\end{proof}

\section{Gerbes over complete discrete valuation rings}

In this section we prove two results on the structure of étale
gerbes over complete discrete valuation rings that will be used in the
proof of Theorem~\ref{thm:generic}.

\subsection{Big and small étale sites} \label{s.BigToSmall}
Let $S$ be a scheme.  We let $\Sch/S$
denote the category of all schemes $T$ equipped with a morphism to
$S$.  As in~\cite{SGA4.2}, we equip $\Sch/S$ with the étale
topology.  Let $\et/S$ denote the full subcategory of $\Sch/S$
consisting of all schemes étale over $S$ (also with the
étale topology).  The site $\Sch/S$ is the \emph{big étale
  site} and the category $\et/S$ is the \emph{small étale
  site}.  We let $S_{\Et}$ denote the category of sheaves on $\Sch/S$
and $S_{\et}$ the category of sheaves on $\et/S$.  Since the obvious
inclusion functor from the small to the big étale site is
continuous, it induces a continuous morphism of sites $u\colon\et/S\arr
\Sch/S$ and thus a morphism $u\colon S_{\Et}\arr S_{\et}$.  Moreover, the
adjunction morphism $F\arr u_*u^*F$ is an isomorphism for $F$ a sheaf
in $S_{\et}$~\cite[VII.4.1]{SGA4.2}.  We can therefore regard
$S_{\et}$ as a full subcategory of $S_{\Et}$.

\begin{definition}
Let $S$ be a scheme. An \emph{étale gerbe} over $S$ is a separated locally finitely presented \dm stack over $S$ that is a gerbe in the étale topology.
\end{definition}

Let $\cX\arr S$ be an étale gerbe over a scheme $S$.  Then, by
definition, there is an étale atlas, i.e., a morphism $U_0\arr \cX$,
where $U_0\arr S$ is surjective, étale and finitely presented over
$S$.  This atlas gives rise to a groupoid $\cG\eqdef
[U_1\eqdef U_0\times_{\cX}U_0\double U_0]$ in which each term is étale
over $S$.  Since $\cX$ is the stackification of $\cG$ which is a
groupoid on the small étale site $S_{\et}$, it follows that
$\cX=u^*\cX'$ for $\cX'$ a gerbe on $S_{\et}$.  In other words, we have
the following proposition.

\begin{proposition}\label{p.BigToSmall}  
  Let $\cX\arr S$ be an étale gerbe over a scheme $S$.  Then there is
  a gerbe $\cX'$ on $S_{\et}$ such that $\cX=u^*\cX'$.
\end{proposition}

If $S$ is a henselian trait 
(i.e., the spectrum of a henselian discrete valuation ring) 
we can do better:

\begin{proposition}\label{p.Hensel}
  Let $S$ be a henselian trait and
  $f\colon T\arr S$ be a surjective étale morphism.  Then there is an open 
  component $T'$ of $T$ such that $f_{|T'}\colon T'\arr S$ is a finite
  étale morphism. 
\end{proposition}

\begin{proof}
 Let $s$ denote the closed point of $S$.  Since $f$ is surjective,
 there exists a $t\in T$ such that $f(t)=s$.  Since $f$ is étale,
 $f$ is quasi-finite at $t$ by~\cite[17.6.1]{EGA4.4}.  Now, it follows
from~\cite[18.5.11]{EGA4.4} that $T'\eqdef\Spec \cO_{T,t}$ is an open
component of $T$ which is finite and étale.
\end{proof}

Now for a scheme $S$, let $\fet/S$ denote the category of finite
étale covers $T\arr S$.  We can consider $\fet/S$ as a site in the
obvious way.  Then the inclusion morphism induces a continuous
morphism of sites $v\colon \et/S\arr \fet/S$.  If $S$ is a henselian trait
with closed point $s$, then the inclusion morphism $i\colon s\arr S$ induces
an equivalence of categories $i^*\colon \fet/S\arr \fet/s$.  Since the
site $\fet/s$ is equivalent to $s_{\et}$, this induces the \emph{specialization
morphism} $\spc\colon S_{\et}\arr s_{\et}$, which is inverse to the
inclusion morphism $i\colon s_{\et}\arr S$; cf.~\cite[p.~89]{SGA72}.  
Let $\tau=\spc\circ u\colon S_{\Et}\arr s_{\et}$.

\begin{corollary}\label{c.ReduceToBG}
  Let $\cX\arr S$ be an étale gerbe over a henselian trait $S$ with
  closed point $s$.  Then there is a gerbe $\cX''$ over  $s_{\et}$ 
  such that $\cX=\tau^*\cX''$.  
\end{corollary}

\begin{proof}
  Since $\cX\arr S$ is an étale gerbe, there is an étale atlas
  $X_0\arr S$ of $\cX$.  By Proposition~\ref{p.Hensel} we may assume 
that $X_0$ is finite over $S$. Then 
$X_{1} \eqdef X_{0}\times_{\cX} X_{0}$ is also finite, 
because $\cX$ is separated, by hypothesis.   Now the equivalence 
of categories $i^*\colon \fet/S\arr \fet/s$ produces an gerbe
$\cX''$ over $s_{\et}$ such that $\cX=\tau^*\cX''$.  
\end{proof}

\subsection{Group extensions and gerbes}

Let $k$ be a field with separable closure $\overline{k}$ and absolute Galois group $G=\Gal(\overline k/k)$.
Let 
\begin{equation}
  \label{e.extension}
  1\arr F\stackrel{i}{\larr} E\stackrel{p}{\larr} G\arr 1
\end{equation}
be an extension of profinite groups with $F$ finite and all maps
continuous.  From this data, we can construct a gerbe $\cX_E$
over $(\Spec k)_{\et}$.   To determine the gerbe it is enough to give
its category of sections over $\Spec L$ where $L/k$ is a finite
separable extension.  Let $K=\{g\in G\, |\, g(\alpha)=\alpha, \alpha\in
L\}$.  Then the objects of the category $\cX_E(L)$ are the solutions
of the embedding problem given by~\eqref{e.extension}.  That is, an
object of $\cX_E(L)$ is a continuous homomorphism $\sigma\colon  K\arr E$
such that $p\circ \sigma (k)=k$ for $k\in K$.   If $s_i\colon K\arr E, i=1,2$ are
two objects in $\cX_E(L)$ then a morphism from $s_1$ to $s_2$ is an
element $f\in F$ such that $fs_1f^{-1}=s_2$; cf.~\cite[p.~581]{DebesDouai}.

By the results of Giraud~\cite[Chapter VIII]{G}, it is easy to
see that any gerbe $\cX\arr \Spec k$ with finite inertia arises from a
sequence~\eqref{e.extension} as above.    We explain how to get the
extension:  Given
$\cX$, we can find a separable Galois extension $L/k$ and an object
$\xi\in \cX(L)$.  This gives an extension of groups $\Aut_{\cX}(\xi)\arr
\Aut_{\Spec k}(\Spec L)=\Gal(L/k)$.  Pulling back this extension via
the map $G=\Gal(k)\arr \Gal(L/k)$ gives the desired 
sequence~\eqref{e.extension}.

Now, suppose that $E$ is as
in~\eqref{e.extension}.  Let $L/k$ be a field extension, which is
separable but not necessarily finite.  Let $\overline L$ denote a fixed
separable closure of $L$ and let $\overline k$ denote the separable closure
of $k$ in $\overline L$.  Then there is an obvious map $r\colon \Gal(\overline L/L)\arr
\Gal(\overline k/k)$.   Let $u\colon (\Spec k)_{\Et}\arr (\Spec k)_{\et}$ denote
the functor of~\S\ref{s.BigToSmall}.   Then $u^*\cX(L)$ has the same
description as in the case where $L$ is a finite extension of $k$.  In
other words, we have the following proposition.

\begin{proposition}\label{p.GerbeCategory}
   Let $L/k$ be a separable extension and let $\cX_E$ be the gerbe defined
  above.  Then the objects of the category $u^*\cX_E(L)$ are the
  morphisms $s\colon \Gal(L)\arr E$ making the following diagram
  commute.
$$
\xymatrix{
            &            &              &  \Gal(L)\ar[ld]_s\ar[d]^r & \\
 1\ar[r] & F\ar[r] & E \ar[r]  & \Gal(k) \ar[r] & 1\\
}
$$
Moreover, if $s_i, i=1,2$ are two objects in $u^*\cX_E(L)$, then the
morphisms from $s_1$ to $s_2$ are the elements $f\in F$ such 
that $fs_1f^{-1}=s_2$. 
\qed
\end{proposition}

\subsection{Splitting the inertia sequence}\label{sub.splitting}

We begin by recalling some results and notation from Serre's chapter
in~\cite{gms}.

Let $A$ be a discrete valuation ring.  Write $S=S_A$ for $\Spec A$,
$s=s_A$ for the closed point in $S$ and $\eta=\eta_A$ for the generic
point.  When $A$ is the only discrete valuation ring under
consideration, we suppress the subscripts.  If $A$ is henselian, then
the choice of a separable closure $k(\overline{\eta})$ of $k(\eta)$
induces a separable closure of $k(s)$ and a map
$\Gal(k(\overline{\eta})) \arr\Gal(k(\overline{s}))$ between 
the absolute Galois
groups.  The kernel of this map is called the \emph{inertia}, written
as $I=I_A$.  If $\chr k(s)=p>0$, then we set $I^{\rmw}=I^{\rmw}_A$ 
equal to the unique $p$-Sylow subgroup of $I$; 
otherwise we set $I^{\rmw}=\{1\}$.  The
group $I^{\rmw}$ is called the \emph{wild inertia}.  The group
$I_{\rmt}=I_{A,t}\eqdef I/I^{\rmw}$ is called the \emph{tame inertia} and the
group $\Gal(k(\eta))_{\rmt}\eqdef \Gal(k(\eta))/I^{\rmw}$ is called the \emph{tame
Galois group}.   We therefore have the following exact sequences:
\begin{align}
  &1\larr I\larr \Gal(k(\eta))\larr \Gal(k(s))\larr 1 \quad \text{and}
\label{e.InertiaExact}\\
  &1\larr I_{\rmt}\larr \Gal(k(\eta))_{\rmt}\larr \Gal(k(s))\larr 1
\label{e.TameExact} \, . 
\end{align}
The sequence~\eqref{e.InertiaExact}
 is called the \emph{inertia exact sequence} and~\eqref{e.TameExact} 
the~\emph{tame inertia exact sequence}.

For each prime $l$, set $\displaystyle\ZZ_l(1)=\projlim
\mmu_{l^n}$ so that 
$$\prod_{l\neq p} \ZZ_l(1)=\projlim_{p \, \nmid \, n}\mmu_n.$$  
Then there is a canonical isomorphism $c\colon I_{\rmt}\arr
\prod_{l\neq p}\ZZ_l(1)$~\cite[p.~17]{gms}.  To explain this isomorphism, let 
$g\in I_{\rmt}$ and let $\pi^{1/n}$ be an $n$-th root of a uniformizing parameter
$\pi\in A$ with $n$ not divisible by $p$.  Then the image of $c(g)$ in $\mmu_n$
is $g(\pi^{1/n})/\pi^{1/n}$.  

\begin{proposition}\label{p.Split}
  Let $A$ be a henselian discrete valuation ring.  Then the
  sequence~\eqref{e.InertiaExact} is split. 
\end{proposition}

The proposition extends Lemma 7.6 in~\cite{gms}, where
$A$ is assumed to be complete. 

\begin{proof}
  Because we need the ideas from the proof, we will repeat Serre's
  argument.  Set $K=k(\eta)$ and $\overline K=k(\overline\eta)$.  Set
  $K_{\rmt}={\overline K}^{I_{\rmt}}$: the maximal tamely ramified 
  extension of $K$.
  Let $\pi$ be a uniformizing parameter in $A$.  Then, for each
  non-negative integer $n$ not divisible by $p$, choose an $n$-th root
  $\pi_n$ of $\pi$ in $\overline K$ such that $\pi_{nm}^m=\pi_n$.  Set
  $\Kram\eqdef K[\pi_n]_{(p \, \nmid \, n)}$.  Then $\Kram$ is totally and
  tamely ramified over $K$.  Moreover any $K_{\rmt}=\Kram\Kunr$.  It
  follows that $\Gal(k(s))$ map be identified with the subgroup of
  elements $g\in \Gal(K)_{\rmt}$ fixing each of the $\pi_n$; cf.~\cite{WeilII}.
This splits the sequence~\eqref{e.TameExact}.

Now, in~\cite{gms}, Serre extends this splitting non-canonically to a
splitting of~\eqref{e.InertiaExact} as follows. Since $k(s)$ has
characteristic $p$, the p-cohomological dimension of $\Gal(k(s))$ is
$\leq 1$; see~\cite{cg}.  Consequently, any
homomorphism $\Gal(k(s))\arr \Gal(K)_{\rmt}$ can be lifted to $\Gal(K)$.
\end{proof}

While the splitting of~\eqref{e.TameExact} is not canonical, we need to
know that it is possible to split two such sequences, associated with
henselian discrete valuation rings $A \subseteq B$, in a compatible way.

\begin{proposition}\label{p.SplitCompat}
Let $A\subseteq B$ be an extension of henselian discrete valuation rings, such that a uniformizing parameter for $A$ is also a uniformizing parameter for $B$.  Then there exist maps $\sigma_B\colon \Gal(k(s_B))\arr \Gal(k(\eta_B))_{\rmt}$ 
(resp.  $\sigma_A\colon \Gal(k(s_A))\arr \Gal(k(\eta_A))_{\rmt}$) 
splitting the tame inertia exact sequence~\eqref{e.TameExact} 
for $B$ (resp. $A$) and such that the diagram
\[
\xymatrix{
\Gal(k(s_B))\ar^{\sigma_B}[r]\ar[d] & \Gal(k(\eta_B))_{\rmt}\ar[d]\\
\Gal(k(s_A))\ar^{\sigma_A}[r]       & \Gal(k(\eta_A))_{\rmt},\\
}
\]
with vertical morphisms given by restriction, commutes.
\end{proposition}
\begin{proof}
  Let $\pi\in A$ be a uniformizing parameter for $A$, and hence for $B$.  For each
  $n$ not divisible by $p=\chr(k(s_A))$, choose an $n$-th root $\pi_n$
  of $\pi$ in $k({\overline\eta}_B)$.  Now, set $\sigma_{B}(k(s_B))=
  \{g\in \Gal(k(\eta_B))_{\rmt}\colon  g(\pi_n)=\pi_n\, \text{for all $n$}\}$ and
  similarly for $A$.  By the proof of Proposition~\ref{p.Split}, this
  defines splitting of the tame inertia sequences.  Moreover, these splittings
lift to splittings of the inertia exact sequence.   
\end{proof}

\begin{remark}\label{p.LiftSplit}
  By the proof of Proposition~\ref{p.Split}, the splittings $\sigma_B$ and $\sigma_A$ in Proposition~\ref{p.SplitCompat} can be lifted to maps $\tilde\sigma_B\colon \Gal(k(s_B))\arr \Gal(k(\eta_B))$  (resp. $\tilde\sigma_A\colon \Gal(k(s_A))\arr \Gal(k(\eta_a))$).  However, since these liftings are non-canonical it is not clear that $\tilde\sigma_B$ and $\tilde\sigma_A$ can be chosen compatibly. 
\end{remark}

\subsection{Tame gerbes and splittings}

The following result is certainly well known; for the sake 
of completeness we supply a short proof.

\begin{proposition}
  Let $\cX\arr S$ be an étale gerbe over a henselian trait, with closed point
  $s$.  Denote by $i\colon s \arr S$ the inclusion of the closed point
  and by $\spc\colon S\arr s$ the specialization map.  Then 
the restriction map 
\[ i^*\colon \cX(S)\arr \cX(s) \]
induces an equivalence of
categories with quasi-inverse given by 
\[ \spc^*\colon \cX(s)\arr \cX(S) \, . \] 
\end{proposition}

\begin{proof}
Since the composite $s\arr S\stackrel{\spc}{\arr} s$ is an auto-equivalence and $\cX$ is obtained by pullback from $\cX_s$, it suffices to show that the 
functor $i^*\colon \cX(S)\arr \cX(s)$ is faithful.  For this, suppose $\xi_i\colon S\arr \cX, i=1,2$ are two objects of $\cX(S)$. Then the sheaf $\Hom(\xi_1,\xi_2)$ is étale over $S$.  Since $S$ is henselian, it follows 
that the sections of $\Hom(\xi_1,\xi_2)$ over $S$ are isomorphic 
(via restriction) to the sections over $s$.  
Thus $i^*\colon \cX(S)\arr \cX(s)$ is fully faithful.
\end{proof}

A Deligne-Mumford stack $\cX\arr S$ is \emph{tame} if, for every
geometric point $\xi\colon \Spec\Omega\arr \cX$, the order of the automorphism
group $\Aut_{\Spec \Omega}(\xi)$ is prime to the characteristic
of $\Omega$.    For tame gerbes over a henselian discrete valuation
ring, we have the following analogue of the splitting in
Proposition~\ref{p.SplitCompat}.

\begin{theorem}\label{t.Commute}
Let $h\colon \spec B \arr \spec A$ be the morphism of henselian traits
induced by an inclusion $A \hookrightarrow B$ of henselian discrete 
valuation rings (here we assume that a uniformizing parameter for $A$ 
is sent to a uniformizing parameter for $B$).
Let $\cX$ be a tame étale gerbe over $\spec A$.  Write
  $j_B\colon \{\eta_B\}\arr \spec B$ (resp. $j_A\colon \{\eta_A\}\arr \spec A$) 
for the inclusion of the generic points.  Then there exist functors
\[ \text{$\tau_A \colon \cX(k(\eta_A))\arr \cX(A)$ and 
$\tau_B \colon \cX(k(\eta_B))\arr \cX(B)$} \]
such that the diagram
 \[ \xymatrix{
 \cX (A)\ar^-{j_A^*}[r]\ar^{h^*}[d] &
 \cX (k(\eta_A))\ar^{\tau_A}[r]\ar^{h^*}[d] & 
    \cX(A) \ar^{h^*}[d]\\
 \cX (B)\ar^-{j_B^*}[r] & 
    \cX(k(\eta_B))\ar^{\tau_B}[r] & 
     \cX(B)\,, }
 \]
commutes (up to natural isomorphism) and the horizontal composites
are isomorphic to the identity.
\end{theorem}
\begin{proof}
  Since $\cX$ is an étale gerbe, there is an extension $E$ as
  in~\eqref{e.extension} with $G=\Gal(k(s_A))$ such that $\cX$ is the
  pull-back of $\cX_E$ to the big étale site over $S_A$.  Since
  $\cX$ is tame, the band, i.e., the group $F$ in~\eqref{e.extension},
  has order prime to $\chr k(s_A)$.

  Now, pick splittings $\sigma_B$ and $\sigma_A$ compatibly, as in
  Proposition~\ref{p.SplitCompat}.  

  We define a functor $\tau_B \colon \cX(k(\eta_B))\arr \cX(B)$ as follows.
  Using Proposition~\ref{p.GerbeCategory} we can identify
  $\cX(k(\eta_B))$ with category of sections $s\colon \Gal(k(\eta_B))\arr
  E$.  Given such a section $s$, the tameness of $E$ implies that
  $s(I^{\rmw})=1$.  Therefore, $s$ induces a map $\Gal(k(\eta_B))_{\rmt}\arr
  \Gal(k(s_B))$, which we will also denote by the symbol $s$.  Let
  $\tau_B(s)$ denote the section
  $s\circ\sigma_B\colon \Gal(k(s_B))\arr E$.  This defines
  $\tau_B$ on the objects in $\cX(k(\eta_B))$.  If we define
  $\tau_A$ in the same way, it is clear that the diagram above
  commutes on objects.  We define $\tau_B$ (resp. $\sigma_A$) on
  morphisms, by setting $\tau_B(f)=f$ (and similarly for $A$).  We
  leave the rest of the verification to the reader.
\end{proof}

\subsection{Genericity}

\begin{theorem}\label{t.GenDVR}
Let $R$ be a discrete valuation ring, $S = \Spec(R)$ and
\[ \cX\arr S \]
a tame étale gerbe. Then
$\ed_{k(s)}\cX_s \leq \ed_{k(\eta)}\cX_{\eta}$,
where $s$ is the closed point of $S$ and $\eta$ is 
the generic point.
\end{theorem}

\begin{proof}
  We may assume without loss of generality that $R$ is complete.  
  Indeed, otherwise replace
  $R$ with its completion at $s$.  The field $k(s)$ does not change,
  but $k(\eta)$ is replaced by a field extension. By
  Proposition~\ref{p.extensions}, the essential dimension of
  $\cX_{k(\eta)}$ does not increase.

  If $R$ is equicharacteristic, then by Cohen's structure theorem, $R=k\ds{t}$ with
  $k=k(s)$. If not, denote by $W\bigl(k(s)\bigr)$ the unique complete discrete valuation ring with residue field $k(s)$ and uniformizing parameter $p$. This is called a Cohen ring of $k(s)$ in \cite[19.8]{ega4-1}. 
If $k(s)$ is perfect then $W\bigl(k(s)\bigr)$ is the ring 
of Witt vectors of $k(s)$, but this is not true in general, and 
$W\bigl(k(s)\bigr)$ is only determined up to a non-canonical isomorphism. 
By \cite[Théorème~19.8.6]{ega4-1}, there 
is a homomorphism $W\bigl(k(s)\bigr) \arr R$ inducing 
the identity on $k(s)$. Since $\cX$ is pulled back 
from $k(s)$ via the specialization map, we can 
replace $R$ by $W\bigl(k(s)\bigr)$.

 Now suppose $b\colon \Spec L\arr \cX_s$ is a morphism from the spectrum of
 a field with $\ed_{k(s)}b=\trdeg_{k(s)}L=\ed\cX_s$.  (Such a
 morphism exists because $\ed\cX_s$ is finite.)   Set $B := L\ds{t}$ 
 if $R$ is equicharacteristic and $B := W(L)$
 otherwise.   In either case, $B$ is a complete discrete valuation
 ring with residue field $L$. In the first case we have a canonical 
embedding $R = k\ds{t} \subseteq L\ds{t} = B$; in the second case, 
again by \cite[Théorème~19.8.6]{ega4-1} (due to Cohen), 
we have a lifting $R = W\bigl(k(s)\bigr) \arr W(L) = B$ 
of the embedding $k(s) \subseteq L$, which is easily seen 
to be injective. Therefore there is a unique morphism
  $\beta = b \circ \spc \colon S_B\arr \cX$ whose specialization 
to the closed point of $B$ coincides with $\xi$. 

Suppose there is a subfield $M$ of $k(\eta_B)$ containing $k(\eta_R)$
such that the following conditions hold:
\begin{enumerate}
\item the restriction $j_B^*\beta$ of $\beta$ to $k(\eta_B)$ 
factors through $M$,
\item $\trdeg_{k(\eta_R)} M<\ed_{k(s)} b$.
\end{enumerate}
Complete $M$ with respect to the discrete valuation induced from
$k(\eta_B)$ and call the resulting complete discrete valuation ring
$A$.  It follows that there is a class $\alpha$ in $\cX(k(\eta_A))$ whose
restriction to $k(\eta_B)$ coincides with  $j_B^*\beta$.   But then, by 
Theorem~\ref{t.Commute},  we have $\beta=h_*\sigma_A(\alpha)$.  
This implies that $b\colon \Spec L\arr \cX_s$ factors through the special
fiber of $A$.  Since the transcendence degree of $k(s_A)$ over $k(s)$
is less than $\ed_{k(s)} b$, this is a contradiction.
\end{proof}

\begin{corollary}
Let $R$ be an equicharacteristic complete discrete
local ring and $\cX\arr \Spec(R)$ be a tame étale gerbe.
Then \[ \ed_{k(s)}\cX_s=\ed_{k(\eta)}  \cX_{\eta} \, , \]
where $s$ denotes the closed point of $\Spec(R)$ and $\eta$ 
denotes the generic point.
\end{corollary}

\begin{proof}
Set $k=k(s)$.  Since $R$ is equicharacteristic, we have $R=k\ds{t}$ and
$\cX_{k(\eta)}$ is the pullback to $k(\eta)$ of $\cX_{k(s)}$ via the
inclusion of $k$ in $k\dr{t}$.   Therefore $\ed_{k(s)}\cX_{k(s)}\geq
\ed_{k(\eta)}\cX_{k(\eta)}$.
The opposite inequality is given by Theorem~\ref{t.GenDVR}.
\end{proof}

\begin{theorem} \label{thm:genericity-gerbe} Suppose that $\cX$ is an
  étale gerbe over a smooth scheme $\bX$ locally of finite type over a
  perfect field $k$. Let $K$ be an extension of $k$,
  $\xi \in \cX(\spec K)$. Then
   \[
   \ed \xi
   \leq \ed_{k(\bX)}\cX_{k(\bX)} + \dim \bX - \codim_{\cX} \xi.
   \]
\end{theorem}

\begin{proof}
We proceed by induction on $\codim\xi$.
If $\codim_{\cX}\xi = 0$, then the morphism $\xi\colon \spec K \arr \cX$ 
is dominant. Hence 
$\xi$ factors through $\cX_{k(\bX)}$, and the result is obvious.

Assume $\codim_{\cX}\xi > 0$. Let $\bY$ be the closure of 
the image of $\spec K$ in $\bX$.  Since we are assuming that
$k$ is perfect, $\bY$ is generically smooth over $\spec k$.
By restricting to a neighborhood of the generic point of $\bY$, we may 
assume that $\bY$ is contained in a smooth hypersurface $\bX'$ 
of $\bX$. Denote 
by $\cY$ and $\cX'$ the inverse images in $\cX$ of $\bY$ and $\cX'$ 
respectively. Set $R = \cO_{\bX,\bY}$ and call $\cX$ the pullback of $\cX$ to 
$R$. Then we can apply Theorem~\ref{t.GenDVR} to the gerbe $\cX_{R} 
\arr \spec R$ and conclude that
   \[
   \ed_{k(\bX')}\cX'_{k(\bX')} \leq 
   \ed_{k(\bX)}\cX_{k(\bX)}.
   \]
Using the inductive hypothesis we have
   \begin{align*}
   \ed\xi &\leq \ed_{k(\bX')}\cX_{k(\bX')} + \dim \bX' - \codim_{\cX'} \xi\\
   &\leq \ed_{k(\bX)}\cX_{k(\bX)} + \dim \bX - 1 - \codim_{\cX'} \xi\\
   &\leq \ed_{k(\bX)}\cX_{k(\bX)} + \dim \bX - \codim_{\cX} \xi.
   \qedhere
   \end{align*}
\end{proof}

\section{A genericity theorem for a smooth \dm stack}
\label{s.generic}

It is easy to see that Theorem~\ref{thm:genericity-gerbe}
fails if $\cX$ is not assumed to be a gerbe. 
In this section we will 
use Theorem~\ref{thm:genericity-gerbe} to
prove the following weaker result
for a wider class of \dm stacks.

Recall that a \dm stack $\cX$ over a field $k$ is \emph{tame} 
if the order of the automorphism group of any object of $\cX$ 
over an algebraically closed field is prime to the characteristic 
of $k$.

\begin{theorem}\label{thm:generic}
Let $\cX$ be a smooth integral tame \dm stack locally of finite type over a perfect field $k$. Then
   \[
   \ed\cX = \ed_{k(\bX)}\cX_{k(\bX)} + \dim \cX.
   \]
\end{theorem}

Here the dimension of $\cX$ is the dimension of the moduli space of any non-empty open substack of $\cX$ with finite inertia.

Before proceeding with the proof, we record two immediate 
corollaries. 

\begin{corollary} \label{cor.generic1}
If $\cX$ is as above and $\cU$ is an open dense substack,
then $\ed_{k}\cM = \ed_{k}\cU$. \qed
\end{corollary}

\begin{corollary} \label{cor.generic2}
If the conditions of the Theorem~\ref{thm:generic} are satisfied, 
and the generic object of $\cX$ has no non-trivial 
automorphisms (i.e., $\cX$ is an orbifold,
in the topologists' terminology), then $\ed_{k}\cX = \dim \cX$.
\end{corollary}

\begin{proof} Here the generic gerbe $\cX_K$ is a scheme,
so $\ed_{K}\cX_K = \dim \cX$.
\end{proof}

\begin{proof}[Proof of Theorem~\ref{thm:generic}]
The inequality 
$\ed\cX \geq \ed_{k(\bX)}\cX_{k(\bX)} + \dim \cX$ 
is obvious: so we only need to show that
   \begin{equation}\label{eq:inequality}
   \ed\xi \leq \ed_{k(\bX)}\cX_{k(\bX)} + \dim \cX
   \end{equation}
for any field extension $L$ of $k$ and any object 
$\xi$ of $\cX(L)$. 

First of all, let us reduce the general result to the case that $\cX$ has finite inertia. The reduction is immediate from the following lemma, that is essentially due to Keel and Mori.

\begin{lemma}[Keel--Mori]\label{lem:keel-mori}
There exists an integral \dm stack with finite inertia $\cX'$, 
together with an étale representable morphism of finite type 
$\cX' \arr \cX$, and a factorization $\spec L \arr \cX' \arr \cX$ 
of the morphism $\spec L \arr \cX$ corresponding to $\xi$.
\end{lemma}

\begin{proof}
We follow an argument due to Conrad. By~\cite[Lemma~2.2]{conrad-keel-mori} 
there exist

\begin{enumeratei}

\item an étale representable morphism $\cW \arr \cX$ 
such that every morphism $\spec L \arr \cX$, where $L$ is a field, 
lifts to $\spec L \arr \cW$, and 

\item a finite flat representable map $Z \arr \cW$, where $Z$ is a scheme. 

\end{enumeratei}

Condition (ii) implies that $\cW$ is a quotient of $Z$ by 
a finite flat equivalence relation $Z \times_{\cW} Z \double Z$, 
which in particular tells us that $\cW$ has finite inertia. 
We can now take $\cX'$ to be a connected component of 
$\cW$ containing a lifting $\spec L \arr \cW$ of $\spec L \arr \cX$.
\end{proof}

Suppose that we have proved the inequality~(\ref{eq:inequality}) whenever $\cX$ has finite inertia. If denote by $\xi'$ the object of $\cX'$ corresponding to a lifting $\spec L \arr \cX'$, we have
   \[
   \ed \xi \leq \ed \xi' \leq \ed_{k(\bX')}\cX'_{k(\bX')}.
   \]
On the other hand, the morphism $\cX'_{k(\bX')} \arr \cX_{k(\bX)}$ 
induced by the étale representable morphism 
$\cX' \arr \cX$ is representable with fibers of dimension~0, hence
   \[
   \ed_{k(\bX')}\cX'_{k(\bX')} = \ed_{k(\bX)}\cX'_{k(\bX')}
      \leq \ed_{k(\bX)}\cX_{k(\bX)}
   \]
by Theorem~\ref{thm.fiber-dimension} (the first equality follows immediately from the fact that the extension $k(\bX) \subseteq k(\bX')$ is finite).

So, in order to prove the inequality~(\ref{eq:inequality}) we may assume that $\cX$ has finite inertia. Denote by $\bY \subseteq \bX$ 
the closure of the image of the composite 
$\spec L \arr \cX \arr \bX$, where 
$\spec L \arr \cX$ corresponds to $\xi$, 
and call $\cY$ the reduced inverse image 
of $\bY$ in $\cX$. Since $k$ is perfect, 
$\cY$ is generically smooth; by restricting 
to a neighborhood of the generic point of $\bY$ 
we may assume that $\cY$ is smooth.

Denote by $\cN \arr \cY$ the normal bundle 
of $\cY$ in $\cX$. Consider the deformation 
to the normal bundle $\phi\colon \cM \arr \PP^{1}_{k}$ 
for the embedding $\cY \subseteq \cX$. This is 
a smooth morphism such that 
$\phi^{-1}\AA^{1}_{k} = 
\cX\times_{\spec k}\AA^{1}_{k}$ and $\phi^{-1}(\infty) = \cN$, 
obtained as an open substack of the blow-up of 
$\cX\times_{\spec k}\PP^{1}_{k}$ along $\cY\times\{\infty\}$ 
(the well-known construction, explained for example 
in \cite[Chapter~5]{fulton}, generalizes immediately 
to algebraic stacks). Denote by $\cM^{0}$ the open 
substack whose geometric points are the geometric 
points of $\cM$ with stabilizer of minimal order (this is well 
defined because $\cM$ has finite inertia).

We claim that $\cM^{0} \cap \cN \neq \emptyset$. This would be evident if $\cX$ 
were a quotient stack $[V/G]$, where $G$ is a finite group of order not 
divisible by the characteristic of $k$, acting linearly on a vector space $V$, 
and $\cY$ were of the form $[X/G]$, where $W$ is a $G$-invariant linear 
subspace of $V$. However, étale locally on $\bX$ every tame \dm stack is a 
quotient $[X/G]$, where $G$ is a finite group of order not divisible by the 
characteristic of $k$ (see, e.g.,~\cite[Lemma 2.2.3]{dan-vistoli02}). 
Since $G$ is tame and $X$ is smooth, it is well known that étale-locally on $
\bX$, the stack $\cX$ has the desired form, and this is enough to prove
the claim.

Set $\cN^{0} \eqdef \cM^{0} \cap \cN$. The object $\xi$ corresponds to a 
dominant morphism $\spec L \arr \cY$. The pullback $\cN\times_{\cY}\spec K$ is 
a vector bundle $V$ over $\spec L$, and the inverse image $\cN^{0}\times_{\cY} 
\spec L$ of $\cN^{0}$ is not empty. We may assume that $L$ is 
infinite; otherwise $\ed\xi = 0$ and there is
nothing to prove. Assuming that $L$ is infinite, 
$\cN^{0}\times_{\cY} \spec L$ has an $L$-rational point, so 
there is a lifting $\spec L \arr \cN^{0}$ of $\spec L \arr \cY$, corresponding 
to an object $\eta$ of $\cN^{0}(\spec L)$. Clearly the essential 
dimension of $\xi$ as an object of $\cX$ is the same as its essential dimension 
as an object of $\cY$, and $\ed\xi \leq \ed\eta$. Let us apply Theorem~
\ref{thm:genericity-gerbe} to the gerbe $\cM^{0}$. The function field of the 
moduli space $\bM$ of $\cM$ is $k(\bX)(t)$, and its generic gerbe is $
\cX_{k(\bX)(t)}$; by Proposition~\ref{p.extensions}, we have $\ed_{k(\bX)(t)}
\cX_{k(\bX)(t)} \leq \ed_{k(\bX)}\cX_{k(\bX)}$. The composite $\spec L \arr 
\cN^{0} \subseteq \cM^{0}$ has codimension at least~$1$, hence we obtain
   \begin{align*}
   \ed \xi &< \ed_{k(\bX)(t)}\cX_{k(\bX)(t)} + \dim\bM\\
   &\leq \ed_{k(\bX)}\cX_{k(\bX)} + \dim\bX + 1.
   \end{align*}
This concludes the proof.
\end{proof}


%

\begin{example} \label{ex1.genericity} 

The following examples show that Corollary~\ref{cor.generic2}
(and thus Corollary~\ref{cor.generic1} and Theorem~\ref{thm:generic})
fail for more general algebraic stacks, such as
(a) singular \dm stacks, (b) non \dm stacks, including
quotient stacks of the form $[W/G]$, where
$W$ is a smooth complex affine variety with an action of
a connected complex reductive linear algebraic group $G$ 
acting on $W$.

\smallskip
(a) Let $r, n \ge 2$ be integers. Assume that 
the characteristic of $k$ is prime to $r$. 
Let $W \subseteq \AA^{n}$ be the Fermat hypersurface 
defined by the equation $x_{1}^{r} + \dots + x_{n}^{r} = 0$ and 
$\Delta \subset \AA^n$ be the union of 
the coordinate hyperplanes defined 
by $x_i = 0$, for $i = 1, \dots, n$.
The group $G := \mmu_{r}^{n}$ acts on $\AA^n$ 
via the formula
\[
(s_{1}, \dots, s_{n})(x_{1}, \dots, x_{n}) = (s_{1}x_{1}, \dots, s_{n}x_{n})
\, , 
\]
leaving $W$ and $\Delta$ invariant. Let $\cX := [W/G]$.
Since the $G$-action on $W \setminus \Delta$ is free,
$\cX$ is generically an affine scheme of dimension $n-1$. 
On the other hand, $[\{0\}/G] \simeq \cB_{k}\mmu_{r}^{n}$ 
is a closed substack of $\cX$ of essential dimension $n$;
hence, $\ed(\cX) \ge n$.
%

\smallskip
(b) Consider the action of $G = \GL_n$ on the affine space $M$ of all
$n \times n$-matrices by multiplication on the left. 
Since $G$ has a dense orbit, and the stabilizer of a non-singular 
matrix in $M$ is trivial, we see that $[M/G]$ is generically 
a scheme of dimension $0$.
On the other hand, let $Y$ be the locus of matrices of rank~$n-1$, which 
is a locally closed subscheme of $M$. There is a surjective 
$\GL_n$-equivariant morphism $Y \arr \PP^{n-1}$, sending 
each matrix of rank $n-1$ to its kernel,
which induces a morphism $[Y/G] \arr \PP^{n-1}$. 
If $L$ is an extension of $\CC$, every $L$-valued 
point of $\PP^{n-1}$ lifts to an $L$-valued point of $Y$.
Hence,
   \[
   \ed {[M/G]} \geq \ed {[Y/G]} \geq n-1 \, .
   \]
As an aside, we remark that a similar argument with $Y$ replaced
by the locus of matrices of rank $r$, shows that the essential 
dimension of $[M/G]$ is in fact the maximum of the dimensions 
of the Grassmannians of $r$-planes in $\CC^{n}$, as $r$ ranges 
between $1$ and $n-1$, which is $n^{2}/4$ 
if n is even, and $(n^{2} - 1)/4$, if $n$ is odd.
\end{example}

\begin{question}
Under what hypotheses does the genericity theorem hold? 
Let $\cX \arr \spec k$ be an integral algebraic stack. 
Using the results of \cite[Chapter~11]{LMB}, one can define 
the generic gerbe $\cX_{K} \to \spec K$ of $\cX$, which 
is an fppf gerbe over a field of finite transcendence degree 
over $k$. What conditions on $\cX$ ensure the equality
   \[
   \ed_{k} \cX = \ed_{K}\cX_{K} + \trdeg_{k}K\,?
   \]
Smoothness seems necessary, as there are counterexamples even for \dm 
stacks with very mild singularities; see Example~\ref{ex1.genericity}(a).
We think that the best result that one can hope for is the following. 
Suppose that $\cX$ is smooth with quasi-affine diagonal, 
and let $\xi\in \cX(\spec L)$ be a point. Assume that 
the automorphism group scheme of $\xi$ over $L$ is linearly 
reductive. Then $\ed \xi \leq \ed_{K}\cX_{K} + \trdeg_{k}K$. 
In particular, if all the automorphism groups are linearly reductive, 
then $\ed \cX = \ed_{K}\cX_{K} + \trdeg_{k}K$.
\end{question}

\section{The essential dimension of $\cM_{g,n}$ for $(g,n) \neq (1,0)$}
\label{s.ed-Mgn}

Recall that the base field $k$ is assumed 
to be of characteristic $0$.

The assertion that $\ed\overline{\cM}_{g,n} = \ed \cM_{g,n}$
whenever $2g-2+n > 0$ is an immediate consequence of 
Corollary~\ref{cor.generic1}.  Moreover, if $g \geq 3$, or $g = 2$ and
$n \geq 1$, or $g = 1$ and $n \geq 2$, then
   \[
    \ed_{k}\cM_{g,n} = \ed_{k}\overline{\cM}_{g,n}
    = 3g - 3 + n.
    \]
Indeed, in all these cases the automorphism group of a generic
object of $\cM_{g,n}$ is trivial, so the generic gerbe is trivial, and
$\ed\cM_{g,n} = \dim \cM_{g,n}$ by Corollary~\ref{cor.generic2}.

The remaining cases of Theorem~\ref{thm.curves}, with the exception of
$(g,n) = (1,0)$, are covered by the following proposition.
The case where $(g, n) = (1, 0)$ requires a separate argument 
which will be carried out in the next section.

\begin{proposition} \label{prop.sect8}
\hfil
\begin{enumeratea}

\item $\ed \cM_{0,1} = 2$, 

\item  $\ed \cM_{0,1} = \ed \, \cM_{0,2} = 0$,

\item  $\ed\cM_{1,1}=2$,

\item  $\ed \cM_{2,0} = 5$.

\end{enumeratea}
\end{proposition}

\begin{proof}
(a) Since $\cM_{0,0}\simeq \cB_{k}\PGL_2$, we have
$\ed\cM_{0,0} = \ed \PGL_2 = 2$, where the last inequality is proved
in~\cite[Lemma 9.4 (c)]{reichstein} (the argument there is valid for 
any field $k$ of characteristic $\ne 2$).

Alternative proof of (a):
The inequality $\ed\cM_{0,0} \leq 2$ holds because 
every smooth curve of genus~$0$ over a field $K$ is 
a conic $C$ in $\PP^{2}_{K}$. After a change of coordinates
in $\PP^2_K$ we may assume that $C$ is given by an equation
of the form $ax^{2} + by^{2} + z^{2} = 0$ for some 
$a$, $b \in K$, and hence descends to the field 
$k(a,b)$ of transcendence degree $\le 2$ over $k$.
The opposite inequality follows from Tsen's theorem.

\smallskip
(b) A smooth curve $C$ of genus $0$ with one or two rational points 
over an extension $K$ of $k$ is isomorphic to $(\PP^{1}_{k}, 0)$ or
$(\PP^{1}_{k}, 0, \infty)$. Hence, it is defined over $k$. 

Alternative proof of (b): $\cM_{0,2} = \cB_{k}\gm$ and $\cM_{0,1} = 
\cB_{k}(\gm \ltimes \GG_{\rma})$, 
and the groups $\gm$ and $\gm \ltimes \GG_{\rma}$ are special 
(and hence have essential dimension $0$).

\smallskip
(c) Let $\cM_{1,1} \arr \AA^1_{k}$ denote the map given by the
  $j$-invariant and let $\cX$ denote the pull-back of $\cM_{1,1}$ to
  the generic point $\Spec k(j)$ of $\AA^1$. Then $\cX$ is banded by
  $\mmu_2$ and is neutral by~\cite[Proposition 1.4 (c)]{Silverman}, and so
  $\ed_k \cX = \ed_{k(j)} \cX + 1 = \ed \cB_{k(j)}  \mmu_{2} + 1 = 2$. 

(d) is a special case of Theorem~\ref{t.hyperelliptic} below,
since $\cH_{2} = \cM_{2,0}$.
\end{proof}

Let $\cH_g$ denote the stack of hyperelliptic curves of genus $g>1$
over a field $k$ of characteristic $0$. This must be defined with some care; defining a family of hyperelliptic curves as a family $C \arr S$ in $\cM_{g,0}$ whose fiber are hyperelliptic curves will not yield an algebraic stack. There are two possibilities.

\begin{enumeratea}

\item One can define $\cH_{g}$ as the closed reduced substack 
of $\cM_{g}$ whose geometric points corresponds to hyperelliptic curves.

\item As in \cite{arsie-vistoli}, an object of $\cH_{g}$ can 
be defined as two morphisms of schemes $C \arr P \arr S$, 
where $P \arr S$ is a Brauer--Severi, 
$C \arr P$ is a flat finite finitely presented morphism 
of constant degree~2, and the composite $C \arr S$ 
is a smooth morphism whose fibers are connected curves of constant genus~$g$.

\end{enumeratea}

We adopt the second definition; $\cH_{g}$ is then a smooth algebraic 
stack of finite type over $k$ (this is shown in \cite{arsie-vistoli}). 
Furthermore,  there is a natural morphism $\cH_{g} \arr \cM_{g,0}$, 
which sends $C \arr P \arr S$ into the composite $C \arr S$. This morphism
is easily seen to be a closed embedding. Hence the two stacks defined 
above are in fact isomorphic.

\begin{theorem}
\label{t.hyperelliptic}
$ \ed \cH_g =
\begin{cases}
 2g   & \text{if $g \ge 3$ is odd,}\\
 2g+1 & \text{if $g \ge 2$ is even.}\\
\end{cases}
$
\end{theorem} 

\begin{proof} Denote by $\bH_{g}$ the moduli space of
$\cH_{g}$; the dimension of $\cH_{g}$ is $2g-1$. Let $K$ be the field
of rational functions on $\bH_{g}$, and denote by $(\cH_{g})_{K}
\eqdef \spec K \times_{\bH_{g}} \cH_{g}$ the generic gerbe of
$\cH_{g}$. From Theorem~\ref{thm:generic} we have
   \[
   \ed \cH_{g} = 2g - 1 + \ed_{K}(\cH_{g})_{K},
   \]
so we need to show that
$\ed_{K}(\cH_{g})_{K}$ is $1$ if $g$ is odd, $2$ if $g$ is
even.  For this we need some standard facts about stacks of hyperelliptic
curves, which we will now recall.

Let $\cD_{g}$ be the stack over $k$ whose objects over a $k$-scheme $S$
are pairs $(P\to S, \Delta)$, where $P \to S$ is a conic bundle (that
is, a Brauer--Severi scheme of relative dimension~$1$), and
$\Delta\subseteq P$ is a Cartier divisor that is étale of
degree~$2g+2$ over $S$. Let $C \xarr{\pi} P \arr S$ be an object of $\cH_{g}$; denote by $\Delta \subseteq P$ the ramification locus of $\pi$. Sending $C \xarr{\pi} P \arr S$ to $(P \arr S,
\Delta)$ gives a morphism $\cH_{g} \arr \cD_{g}$. Recall the usual
description of ramified double covers: if we split $\pi_{*}\cO_{C}$ as
$\cO_{P} \oplus L$, where $L$ is the part of trace~$0$, then
multiplication yields an isomorphism $L^{\otimes 2} \simeq
\cO_{P}(-\Delta)$. Conversely, given an object $(P \to S, \Delta)$ of
$\cD_{g}(S)$ and a line bundle $L$ on $P$, with an isomorphism
$L^{\otimes 2} \simeq \cO_{P}(-\Delta)$, the direct sum $\cO_{P}
\oplus L$ has an algebra structure, whose relative spectrum is a
smooth curve $C\to S$ with a flat map $C \arr P$ of degree~$2$.

The morphism $\cH_{g} \arr \bH_{g}$ factors through $\cD_{g}$, and the
morphism $\cD_{g} \arr \bH_{g}$ is an isomorphism over the non-empty
locus of divisors on a curve of genus~$0$ with no non-trivial
automorphisms (this is non-empty because $g \geq 2$, hence $2g+2 \geq
5$). Denote by $(P \arr \spec K, \Delta)$ the object of $\cD_{g}(\spec K)$
corresponding to the generic point $\spec K \arr \bH_{g}$. It is well-known 
that $P(K) = \emptyset$; we give a proof for lack of a suitable reference.

Let $C$ be a conic without rational points defined over some extension $L$ of 
$k$. Let $V$ be the $L$-vector space $\H^{0}(C, \omega_{C/L}^{-(g+1)})$;
denote the function field of $V$ by $F = L(V)$.
Then there is a tautological section $\sigma$ of $\H^{0}
(C_{F}, \omega_{C_{F}/F}^{-(g+1)}) = \H^{0}(C, \omega_{C/L}^{-(g+1)}) 
\otimes_{L} F$. Note that $C_{F}(F) = \emptyset$, because the extension $L 
\subseteq F$ is purely transcendental. The zero scheme of $\sigma$ is a divisor 
on $C_{F}$ that is étale over $\spec F$, and defines a morphism $C_{F} \arr 
\cD_{g}$. This morphism is clearly dominant: so $K \subseteq F$, and $C_{F} = P 
\times_{\spec L} \spec F$. Since $C_{F}(F) = \emptyset$ we have $P(K) = 
\emptyset$, as claimed.

By the description
above, the gerbe $(\cH_{g})_{K}$ is the stack of square roots of
$\cO_{P}(-\Delta)$, which is banded by $\mmu_{2}$. When $g$ is odd then
there exists a line bundle of degree $g+1$ on $P$, whose square is
isomorphic to $\cO_{P}(-\Delta)$; this gives a section of
$(\cH_{g})_{K}$, which is therefore isomorphic to $\cB_{K}\mmu_{2}$,
whose essential dimension over $K$ is $1$. If $g$ is even then
such a section does not exist, and the stack is isomorphic to the
stack of square roots of the relative dualizing sheaf $\omega_{P/K}$ (since $\cO_{P/K}(-\Delta) \simeq \omega_{P/K}^{g+1}$, and $g+1$ is odd), whose class in $\H^{2}(K,
\mmu_{2})$ represents the image in $\H^{2}(K, \mmu_{2})$ of the class
$[P]$ in $\H^{1}(K, \PGL_{2})$ under the non-abelian boundary map
$\H^{1}(K, \PGL_{2}) \arr \H^{2}(K, \mmu_{2})$. According to
Theorem~\ref{t.edGerbe} its essential dimension is the index of $[P]$,
which equals $2$.
\end{proof}

The results above apply to more than stable curves. Assume that 
we are in the stable range $2g - 2 + n > 0$.
Denote by $\fM_{g,n}$ the stack of all reduced $n$-pointed 
\lci curves of genus $g$. This is the algebraic stack 
over $\spec k$ whose objects over a $k$-scheme $T$ are 
finitely presented proper flat morphisms 
$C \arr T$, where $C$ is an algebraic space, whose geometric 
fibers are connected reduced \lci curves of genus $g$, 
together with $n$ sections $T \arr C$ whose images are 
contained in the smooth locus of $C \arr T$. We do not require 
the sections to be disjoint.

The stack $\fM_{g,n}$ contains $\cM_{g,n}$ as an open substack. 
By standard results in deformation theory, every reduced 
\lci curve is unobstructed, and is a limit of smooth curves.
Furthermore there is no obstruction to extending the sections, 
since these map into the smooth locus. Therefore 
$\fM_{g,n}$ is smooth and connected, and $\cM_{g,n}$ is 
dense in $\fM_{g,n}$. However, the stack $\fM_{g,n}$ is very 
large (it is certainly not of finite type), and in fact it is 
very easy to see that its essential dimension is infinite. However, 
consider the open substack $\fM_{g,n}^{\mathrm{fin}}$ consisting 
of objects whose automorphism group is finite. Then 
$\fM^{\mathrm{fin}}_{g,n}$ is a \dm stack, and 
Theorem~\ref{thm:generic} applies to it. Thus we 
get the following strengthened form of Theorem~\ref{thm.curves}
(under the assumption that $2g-2+n > 0$).

\begin{theorem} \label{thm.finite}
If $2g-2+n > 0$ and the characteristic of $k$ is $0$, then
\[ \ed\fM^{\mathrm{fin}}_{g,n} = 
   \begin{cases} 
   2         & \text{if }(g,n)=(1,1), \\
   5         & \text{if }(g,n)=(2,0),\\
   3g-3 + n  & \text{otherwise}.
\end{cases}
\]
\end{theorem}

It is not hard to show that $\fM^{\mathrm{fin}}_{g,n}$ does not have 
finite inertia.

\section{Tate curves and the essential dimension of $\cM_{1,0}$}
\label{s.Tate}

In this section we will finish the
proof of Theorem~\ref{thm.curves} by showing
that $\ed\cM_{1,0}=+\infty$.  

We remark that the moduli stack $\cM_{1,0}$ of genus $1$ curves
should not be confused with the moduli stack $\cM_{1,1}$ 
of elliptic curves.
The objects of $\cM_{1,0}$ are torsors for elliptic curves, 
where as the objects of $\cM_{1, 1}$ are elliptic curves 
themselves. The stack $\cM_{1,1}$ is \dm and, as we saw 
in the last section, its essential dimension is $2$. 
The stack $\cM_{1, 0}$ is not \dm, and we will now
show that its essential dimension is $\infty$.

 Let $R$ be a complete discrete valuation ring with
  function field $K$ and uniformizing parameter $q$.  For simplicity,
  we will assume that $\chr K=0$.  Let $E=E_q/K$ denote
  the Tate curve over $K$~\cite[\S 4]{Silverman}.  This is an elliptic
  curve over $K$ with the property that, for every finite field
  extension $L/K$, $E(L)\cong L^*/q^{\ZZ}$.  It follows that the 
  kernel  $E[n]$ of multiplication by an integer $n>0$ fits
  canonically into a short exact sequence
\begin{equation}
\label{t.ses}
0 \arr \mmu_n \arr E[n] \arr \ZZ/n \arr 0.
\end{equation}
Let $\partial\colon\H^0(K,\ZZ/n) \arr \H^1(K,\mmu_n)$ denote the connecting
homomorphism.  Then it is well-known (and easy to see) that
$\partial(1)=q\in \H^1(K,\mmu_n)\cong K^*/(K^*)^n$.  
%
%

\begin{lemma}
\label{l.TateTorsion}
  Let $E=E_q$ be a Tate curve as above and let $l$ be a
  prime integer not equal to $\chr R/q$.  Then, for any integer $n>0$, 
  \begin{equation*}
    \ed E[l^n] = l^n.
  \end{equation*}
\end{lemma}
\begin{proof}
First observe that $E[l^n]$ admits an 
$l^n$-dimensional generically free representation 
$V=\Ind_{\mmu_{l^n}}^{E[l^n]} \chi$, over $K$,
where $\chi\colon\mmu_{l^n} \arr \GG_m$ is the tautological character.
Thus, 
\[ \ed \cB E[l^n] \le \dim(V) = l^n \, ; \]
see~\cite[Theorem 3.1]{bur} or~\cite[Proposition 4.11]{bf1}.

It remains to show that 
\begin{equation} \label{e.Tate}
\ed E[l^n]\geq l^n \, .
\end{equation}
Let $R'\eqdef R[1^{1/l^n}]$ with fraction field $K'=K[1^{1/l^n}]$.  Since
$l$ is prime to the residue characteristic, $R'$ is a complete
discrete valuation ring, and the Tate curve $E_q/K'$ is the pullback
to $K'$ of $E_q/K$.  Since $\ed (E_q/K')\leq \ed (E_q/K)$, it suffices
to prove the lemma with $K'$ replacing $K$.  In other words, it
suffices to prove the inequality~\eqref{e.Tate}
under the assumption that $K$ contains the $l^n$-th roots of unity.

In that case, we can pick a primitive $l^n$-th root of unity $\zeta$
and write $\mmu_{l^n}=\ZZ/l^n$.   Let $L=K(t)$ and consider the 
class $(t)\in \H^1(L,\mmu_{l^n})=L^*/(L^*)^n$.  
It is not difficult to see that 
$$
\partial(t)=q\cup (t).
$$
Since the map $\alpha\mapsto \alpha\cup (t)$ is injective by
cohomological purity, the exponent of $q\cup (t)$ is $l^n$.
Therefore $\ind(q\cup (t))=l^n$.  Then, since $\dim\ZZ/l^n=0$,
Corollary~\ref{cor.lower-bound2}, applied to 
the sequence~\eqref{t.ses}
implies that $\ed \cB E[l^n]\geq l^n$, as claimed. 
\end{proof}

\begin{theorem}
\label{t.Tate} Let $E=E_q$ denote the Tate curve over a field $K$ as
above.  Then  $\ed_K E=+\infty$.
\end{theorem}

\begin{proof}
  For each prime power $l^n$, the morphism $\cB E[l^n] \arr \cB E$ is
  representable of fiber dimension $1$.  By Theorem~\ref{thm.fiber-dimension}
\[ \ed E \geq \ed \cB E[l^n] =l^n-1 \]
for every $n \ge 1$.
\end{proof}

\begin{remark} It is shown in~\cite{bs} that if
$A$ is an abelian variety over $k$ and $k$ is 
a number field then $\ed_k A = +\infty$.
On the other hand, if $k = \mathbb{C}$ is the field of complex 
numbers then $\ed_{\mathbb{C}}(A)= 2 \dim(A)$; see ~\cite{brosnan}. 
\end{remark}

Now we can complete the proof of Theorem~\ref{thm.curves}.

\begin{theorem}
  \label{t.m1}  Let $k$ be a field. Then $\ed_{k}\cM_{1,0}=+\infty$.   
\end{theorem}
\begin{proof}
Set $F=k\dr{t}$.  By Proposition~\ref{p.extensions} 
$\ed_F (\cM_{1,0}\otimes_kF) \le \ed_k \cM_{1,0}$, so
it suffices to show that $\ed_{F} (\cM_{1,0}\otimes_k F)$ is infinite. 
  Consider the morphism $\cM_{1,0} \arr \cM_{1,1}$ which sends a genus $1$
  curve to its Jacobian.  Let $E$ denote the Tate
  elliptic curve over $F$, which is classified by a morphism $\Spec
  F \arr \cM_{1,1}$.  We have a Cartesian diagram:
\[ \xymatrix{
    \cB_{k}E\ar[r]\ar[d] & \cM_{1,0}\otimes_k F\ar[d]\\
    \Spec F\ar[r]  & \cM_{1,1}\otimes_k F.
}
  \]
It follows that the morphism $\cB_{k}E \arr \cM_{1,0}$ is representable, with 
fibers of dimension $\leq 0$. Applying
Theorem~\ref{thm.fiber-dimension} once again, we see that
 \[ +\infty=\ed \cB_{F}E\leq \ed_F (\cM_{1,0}\otimes_kF) \le 
\ed_k \cM_{1,0} \, ,
\]
as desired.
\end{proof}

%
 
 
 
 
 
\let\oldmarginpar\marginpar 
\renewcommand\marginpar[1]{\-\oldmarginpar{\raggedright\small\sf #1}} 
 
 
  
 
 
\newcommand{\nc}{\newcommand} 
 
\nc{\rnc}{\renewcommand} 
 
\nc{\bs}{\backslash} 
\nc{\te}{\otimes} 
\nc{\lf}{\lfloor} 
\nc{\rf}{\rfloor} 
\nc{\lc}{\lceil}  
\nc{\rc}{\rceil} 
\nc{\lr}{\longrightarrow} 
\nc{\sr}{\stackrel} 
\nc{\dar}{\dashrightarrow} 
\nc{\thra}{\twoheadrightarrow} 
\nc{\from}{\leftarrow}

\nc{\mc}{\mathcal} 
\nc{\mb}{\mathbb} 
\nc{\mf}{\mathbf} 
\nc{\mr}{\mathrm} 
 
\rnc{\bP}{\mathbb{P}} 
\rnc{\P}{\mathbb{P}} 
\nc{\Q}{\mathbb{Q}} 
\nc{\Z}{\mathbb{Z}} 
\nc{\C}{\mathbb{C}} 
\nc{\R}{\mathbb{R}} 
\nc{\A}{\mathbb{A}} 
\nc{\V}{\mathbb{V}} 
\nc{\W}{\mathbb{W}} 
\nc{\N}{\mathbb{N}} 
\nc{\G}{\mathbb{G}} 
\nc{\F}{\mathbb{F}} 
 
\nc{\aff}{{\A}^1} 
\nc{\naive}{\!\sim_n} 
\rnc{\Spec}{\mr{Spec}} 
 
\nc{\omx}{\omega_X} 
 
\nc{\ep}{\epsilon} 
\nc{\ve}{\varepsilon} 
 
\nc{\wt}{\widetilde} 
\nc{\wh}{\widehat} 
\nc{\ol}{\overline} 
\rnc{\sl}{\shoveleft} 
 
 
\nc{\Pic}{\operatorname{Pic}} 
\rnc{\Br}{\operatorname{Br}} 
 
\rnc{\chr}{\operatorname{char}}

\numberwithin{equation}{section} 
 

 
\section{Appendix: Essential dimension of moduli of abelian varieties.
By Najmuddin Fakhruddin} 
 
In Theorem~\ref{thm.curves}, Brosnan, Reichstein and Vistoli compute 
the essential dimension of various moduli stacks of curves as an 
application of their ``genericity theorem'' for the essential 
dimension of smooth and tame Deligne--Mumford stacks. Here we use this 
theorem to compute the essential dimension of some stacks of abelian 
varieties. Our main result is: 
 
\begin{thm} 
\label{thm:mainab} 
Let $g \geq 1$ be an integer, $\mc{A}_g$ the stack of $g$-dimensional 
principally polarised abelian varieties over a field $K$ and 
$\mc{B}_g$ the stack of all $g$-dimensional abelian varieties over 
$K$. 
\begin{enumerate} 
\item If $\chr(K) = 0$ then $\ed \mc{A}_g = g(g+1)/2 + 2^a = \ed 
  \mc{B}_g$, where $2^a$ is the largest power of $2$ dividing $g$. 
\item If $\chr(K) = p > 0$ and $p \nmid |Sp_{2g}(\Z/\ell\Z)|$ for some 
  prime $\ell > 2$ then $\ed \mc{A}_g = g(g+1)/2 + 2^a$ with $a$ as 
  above. 
\end{enumerate} 
\end{thm} 
 
For $g$ odd this result is due to Miles Reid. 
 
We do not know if the restriction on $\chr(K)$ is really necessary; in 
Theorem \ref{thm:eda1} we show by elementary methods that for $g=1$ it 
is not. 
 
\smallskip 
 
The main ingredient in the proof, aside from Theorem~\ref{thm:generic}, 
is: 
\begin{thm} 
\label{thm:double} 
Let $K$ be a field with $\chr(K) \neq 2$ and let $\mc{R}_g$ be the 
moduli stack of (connected) etale double covers of smooth projective 
curves of genus $g$ with $g >2$ over $K$. Then the index of the 
generic gerbe of $\mc{R}_g$ is $2^b$, where $2^b$ is the largest power 
of $2$ dividing $g-1$. Furthermore, if $\mc{R}_g$ is tame then $\ed 
\mc{R}_g = 3g -3 + 2^b$. 
\end{thm} 
 
 
\smallskip 
 
The two theorems stated above are connected via the Prym map 
$\mc{R}_{g+1} \to \mc{A}_g$. 
 
\subsection{} 
 
It is easy to get an upper bound on the index of the generic gerbe of 
$\mc{A}_{g,d}$ over any field. This gives an upper bound on the essential 
dimension whenever $\mc{A}_{g,d}$ is smooth and tame. 
 
\begin{prop} 
\label{prop:pall} 
Let $d>0$ be an integer and $\mc{A}_{g,d}$ the moduli stack of abelian 
varieties with a polarisation of degree $d$ over $K$. Then 
\begin{enumerate} 
\item The index of the generic gerbe of each irreducible component of 
  $\mc{A}_{g,d}$ is $\leq 2^a$ if $\chr(K) \neq 2$. If $\chr(K) =2$ 
  then the generic gerbes are all trivial. 
\item If $p = \chr(K) > 0$  assume that $p \nmid d \cdot 
  |GL_{2g}(\Z/\ell \Z)|$ (or if $d=1$, $p \nmid |Sp_{2g}(\Z/\ell 
  \Z)|$) for some prime $\ell >2$. Then $\ed \mc{A}_{g,d} \leq 
  g(g+1)/2 + 2^a$. 
\end{enumerate} 
\end{prop} 
 
\begin{proof} 
  For any $g,d$, $\mc{A}_{g,d}$ is a Deligne--Mumford stack over $K$ 
  with each irreducible component of dimension $g(g+1)/2$ (see 
  \cite{oort-norman} for the case $\chr(K)| d$).  It is a consequence 
  of a theorem of Grothendieck \cite[Theorem 2.4.1]{oort-lifting}, 
  that if $p \nmid d$ then $\mc{A}_{g,d}$ is smooth. Furthermore, if 
  $p \nmid |GL_{2g}(\Z/\ell \Z)|$ (or if $d=1$, $p \nmid 
  |Sp_{2g}(\Z/\ell \Z)|$) for $\ell$ as above then $\mc{A}_{g,d}$ is 
  also tame.  By Theorem~\ref{thm:generic} we see that (2) follows 
  from (1). 
 
  Assume $\chr(K) \neq 2$. The generic gerbe is a gerbe banded by 
  $\Z/2\Z = \mu_2$ so the index is a power of $2$. The Lie algebra 
  ${Lie}_{g,d}$ of the universal family of abelian varieties over 
  $\mc{A}_{g,d}$ is a vector bundle of rank $g$ on which the 
  automorphism $x\mapsto -x$ of the universal family induces 
  multiplication by $-1$. So ${Lie}_{g,d}$ gives rise to a 
  \emph{twisted} sheaf (see e.g. \cite[Section 3]{lieblich-twisted}) 
  on the generic gerbe of each component, hence the index divides 
  $g$. We conclude that the index divides the largest power of $2$ 
  dividing $g$ i.e. $2^a$. 
 
  For any field $L$ of characteristic $2$,  $H^2(L, \Z/2\Z) = 0$ so the 
  generic gerbes above are all trivial if $\chr(K) =2$. 
\end{proof} 
 
If $g$ is odd then it follows that $\ed \mc{A}_{g,d} = g(g+1)/2$ 
whenever $\mc{A}_{g,d}$ is tame and smooth; this was first proved by 
Miles Reid using Kummer varieties. For even $g$ we now use Theorem 
\ref{thm:double}, which we will prove later, to complete the proof of 
Theorem \ref{thm:mainab}. 
 
\begin{proof}[Proof of Theorem \ref{thm:double} implies Theorem 
  \ref{thm:mainab}] 
  We may assume that $g > 1$ since it is known that if $g=1$ then 
  $\ed \mc{A}_g ( = \mc{B}_g) = 2$ (by Theorem~\ref{thm.curves} or Section \ref{sec:a1}). 
 
  We first recall the construction of the Prym map $P: \mc{R}_{g+1} 
  \to \mc{A}_g$. 
 
  Let $f:X \to S$ be a family of smooth projective curves of genus 
  $g+1$ and let $\pi:Y \to X$ be a finite etale double cover (so that 
  the fibres of the composite morphism $f':Y \to S$ are smooth 
  projective curves of genus $2g +1$). Let $\Pic^0_{X/S}, 
  \Pic^0_{Y/S}$ be the corresponding relative Jacobians and let $N : 
  \Pic^0_{Y/S} \to \Pic^0_{X/S}$ be the norm map. The identity 
  component of the kernel of $N$ is an abelian scheme $Prym(Y/X)$ over 
  $S$ of relative dimension $g$ and the involution of $Y$ over $X$ 
  induces an automorphism of $\Pic^0_{Y/S}$ which restricts to 
  multiplication by $-1$ on $Prym(Y/X)$. Furthermore, the canonical 
  principal polarisation on $\Pic^0_{Y/S}$ restricts to $2\lambda$, 
  where $\lambda$ is a principal polarisation on $Prym(Y/X)$. Then $P$ 
  is given by sending $(f:X \to S, \pi:Y\to X)$ to $(Prym(Y/X) \to S, 
  \lambda)$.  The coarse moduli space $\mf{R}_{g+1}$ of $\mc{R}_{g+1}$ 
  is an irreducible variety and $P$ induces a morphism, which we also 
  denote by $P$, $\mf{R}_{g+1} \to \mf{A}_g$. 
 
  Let $\mf{A}_g'$ be the open subvariety of $\mf{A}_g$ corresponding 
  to principally polarised abelian varities $A$ with $Aut(A) = \{ \pm 
  Id \}$. Then $\mc{A}_g |_{\mf{A}_g'} \to \mf{A}_g'$ is a $\mu_2$ 
  gerbe.  Since $P(\mf{R}_{g+1}) \cap \mf{A}_g' \neq \emptyset$ it 
  follows that the generic gerbe of $\mc{R}_{g+1}$ is isomorphic to 
  $\mc{A}_g \times_{\mf{A}_g} \Spec \ K(\mf{R}_{g+1})$.  Since the 
  index at the generic point of an element of the Brauer groups of a 
  smooth variety is greater than or equal to the index at any other 
  point, it follows that the index of the generic gerbe of $\mc{A}_g$ 
  is greater than or equal to the index of the generic gerbe of 
  $\mc{R}_{g+1}$. By Theorem \ref{thm:double} the latter index is $2^a$ 
  and then using Proposition \ref{prop:pall} we deduce the first 
  equality of Theorem \ref{thm:mainab} (1) and also (2), since 
  $\mc{A}_g$ is tame whenever $p \nmid |Sp_{2g}(\Z/\ell \Z)|$ for some 
  prime $\ell \neq 2$. 
 
  Now suppose $\chr(K) = 0$ and let $A$ be any abelian variety of 
  dimension $g$ over an extension field $L$ of $K$. Since $A$ is 
  projective, it follows that $A$ has a polarisation of degree $d$ for 
  some $d>0$ and hence corresponds to an object of $\mc{A}_{g,d}(L)$. 
  By Proposition \ref{prop:pall}, it follows that $A$ together with 
  its polarisation can be defined over a field of transcendence degree 
  $\leq g(g+1)/2 + 2^a$ over $K$, hence $\ed \mc{B}_g \leq g(g+1)/2 = 
  2^a$. A principally polarised abelian variety $A$ over $L$ such that 
  the image of $\Spec \ L$ in $\mf{A}_g$ is the generic point has a 
  unique polarisation which is defined whenever the abelian variety is 
  defined.  It then follows from the previous paragraph that there 
  exists an abelian variety defined over an extension of transcendence 
  degree $g(g+1) + 2^a$ over $K$ which cannot be defined over a 
  subextension of lesser transcendence degree. This proves the second 
  equality of Theorem \ref{thm:mainab} (1). 
\end{proof}

\subsection{} 
\label{sec:delta} 
For any morphism $f:X \to S$, we denote by $\Pic_{X/S}$ the relative 
Picard functor \cite[Chapter 8]{BLR}.  If $\Pic_{X/S}$ is 
representable we use the same notation to denote the representing 
scheme and if $S = \Spec(K)$ is a field we drop it from the notation. 
 
We recall from \cite[Chapter 8, Proposition 4]{BLR} that if $f$ is 
proper and cohomologically flat in dimension $0$, then for any 
$S$-scheme $T$ we have a canonical exact sequence 
\begin{equation} 
\label{seq:relpic} 
0 \to \Pic(T) \to \Pic(X\times_S T) \to \Pic_{X/S}(T) \sr{\delta}{\to} \Br(T) \to \Br(X \times T)  
\end{equation} 
so $\delta(\tau) \in \Br(T)$, for $\tau \in \Pic_{X/S} (T)$, is the 
obstruction to the existence of a line bundle $\mc{L}$ on $X \times_S 
T$ representing $\tau$. 
 
If $X$ is a smooth projective curve over a field, then using the 
morphisms $Sym^d(X) \to \Pic^d_X$ for $d > 0$, the Riemann--Roch 
theorem and Serre duality one sees that the index of $\delta(\tau)$ 
divides $\chi(\tau) = \deg(\tau) + 1 - g$. Since $\delta$ is a 
homomorphism it follows that if $\tau$ is of order $m$ then the order 
of $\delta(\tau)$ divides $m$. We deduce that in this case the index 
of $\delta(\tau)$ divides the largest integer dividing $ g - 1$ all of 
whose prime divisors also divide $m$. Note that if $g=1$ then we do 
not get any bound on the index 
 
\subsection{} 
 
Let $A$ be an abelian variety over a field $K$, let $\tau \in 
\Pic^0_X(K)$, let $\theta \in H^1(K,A)$ and let $P$ be the $A$-torsor 
corresponding to $\theta$. Since $\Pic^0_P$ is canonically isomorphic to 
$\Pic^0_A$, we may view $\tau$ as an element $\tau_P$ of 
$\Pic^0_P(K)$. 
\begin{lem} 
\label{lem:ext} 
With the notation as above, the subgroups of $\Br(K)$ generated by 
$\delta(\tau_P)$ and $\partial(\theta)$ are equal, where $\partial$ is 
the boundary map in the long exact sequence of Galois cohomology 
corresponding to the extension of commutative group schemes 
\[ 
1 \to \G_m \to S \to A \to 0 
\] 
associated to $\tau$ via the isomorphism $\Pic^0_A(K) = Ext^1(A, \G_m)$. 
\end{lem} 
 
\begin{proof} 
  We first remark that as a $\G_m$ bundle on $A$, $S$ is just the 
  complement of the zero section of $\mc{L}$, where $\mc{L}$ is the 
  line bundle on $A$ corresponding to $\tau$ (see e.g.~\cite[Theorem 
  1, p.225]{av}). 
   
  Now let $L$ be any field extension of $K$. If $\delta(\tau_p) = 0$ 
  in $\Br(L)$ then $\tau_P$ is represented by a line bundle $\mc{L}$ 
  on $P_L$. Using the remark above, one sees that $Q$, the complement 
  of the zero section in $\mc{L}$, is an 
  $S_L$-torsor such that $Q \times_{S_L} A_L = P_L$. This implies that 
  $\partial(\theta) = 0$ in $\Br(L)$. Conversely, if $\partial(\theta) 
  = 0$ in $\Br(L)$ then there is a (unique) $S_L$-torsor $Q$ such that 
  $Q \times_{S_L} A_L = P_L$. This gives a $\G_m$ bundle over $P_L$ 
  and hence a line bundle on $P_L$ which represents $\tau_P$, so we 
  must have $\delta(\tau_p) = 0$ in $\Br(L)$. 
 
  It follows that the splitting fields of $\partial(\theta)$ and 
  $\delta(\tau_P)$ are the same, hence the two elements must generate 
  the same subgroup in $\Br(K)$. 
\end{proof} 
 
It is very likely that $\delta(\tau_P)$ and $\partial(\theta)$ are equal, 
at least upto sign, but we shall not need this.

\subsection{} 
 
 
Given a smooth projective curve over a field $K$ and an element $\tau$ 
of $\Pic_X(K)$, one may ask how large the index of $\delta(\tau)$ can 
be. In the case that $\tau$ is torsion, the theorem below shows that 
the best upper bound on the index which is valid over all fields is 
the one given in Section \ref{sec:delta}. 
 
\begin{thm} 
\label{thm:torsion} 
Let $g > 0$ be an integer, $n >0$ an integer such that $n$ divides $ 
g-1$ and $\chr(K)\nmid n$, and $m> 0$ such that $m|n$ and $m,n$ have 
the same prime factors.  Then there exists an extension $L$ of $K$, a 
smooth projective curve $X$ of genus $g$ over $L$ with 
$Aut(X_{\ol{L}}) = \{Id\}$ if $g>2$, and an element $\tau$ of order $m$ 
in $\Pic_X(L)$ such that the index of $\delta(\tau)$ is $n$. 
\end{thm} 
 
The theorem for all $g$ will be deduced from the slightly stronger 
result below for $g=1$. 
 
\begin{prop} 
\label{prop:ell} 
Let $n>0$ be an integer such that $\chr(K)\nmid n $ and $m>0$ such 
that $m|n$ and $m,n$ have the same prime factors.  Then there exists 
an extension $M$ of $K$, a smooth projective geometrically irreducible 
curve $P$ of genus $1$ over $M$ and an element $\sigma$ of order 
$m$ in $\Pic_P(M)$ such that the index of $\delta(\sigma)$ is 
$n$. Furthermore, there exists an extension $M'$ of $M$ of degree 
$n$ such that $P(M')$ is infinite. 
\end{prop}

\begin{proof} 
  We first replace $K$ by $\ol{K}(s,t)$ where $s,t$ are inderminates. 
  We fix an isomorphism $\mu_n \cong \Z/n \Z$ which we use to identify 
  all $\mu_n^{\otimes i}$, $i \in \Z$. For the elements $(s),(t) \in 
  H^1(K,\mu_n)$ consider $\alpha = (s) \cup (t) \in H^2(K, 
  \mu_{n}^{\otimes 2}) \cong H^2(K, \mu_{n}) = {}_n\Br(K)$. It is well 
  known and easy to see that this element of $\Br(K)$ has both order 
  and index equal to $n$.  Let $K'$ be the function field of the 
  Brauer--Severi variety corresponding to the division algebra over 
  $K$ representing $m\alpha$. By a theorem of Amitsur \cite[Theorem 
  9.3]{amitsur-generic} the image of $\alpha$ in $\Br(K')$ has order 
  $m$ and by a theorem of Schofield and Van den Bergh \cite[Theorem 
  2.1]{schofield-bergh} its index is still $n$. 
 
  Let $M$ be the field of Laurent series $K'((q))$ and let $E$ be the 
  Tate elliptic curve over $M$ associated to the element $sq^{n} 
  \in M^{\times}$. For any finite extension $M'$ of $M$ there is a 
  canonical Galois equivariant isomorphism 
  \[ 
  E(M') \cong {M'}^{\times}/\langle sq^{n}\rangle \ . 
  \] 
  From this we get a canonical exact sequence 
  \[ 
  1 \to \mu_{n} \to E[n] \to \Z/n \Z \to 0 \ 
  \] 
  where $1 \in \Z/n \Z$ is the image of any $n$'th root of $sq^{n}$ in 
  $\ol{M'}$. 
  For any $\phi \in H^1(M, \Z/ n \Z)$, one easily checks using the 
  definitions that $\partial(\phi) \in H^2(M, \mu_n) = H^2(M, \mu_n 
  \otimes \Z/n\Z)$ is equal to $(s) \cup \phi$, where $\partial$ 
  denotes the boundary map in the long exact sequence of Galois 
  cohomology associated to the above short exact sequence of Galois 
  modules. It follows that if we identify $(t) \in H^1(M, \mu_{n})$ 
  with an element of $H^1(M, \Z/n \Z)$ using our chosen isomorphism 
  $\Z/n \cong \mu_{n}$, then $\beta :=\partial((t)) = (s) \cup (t)\in 
  H^2(K, \mu_{n}) \subset \Br(M)$.  Thus $\beta$ also has order $m$ 
  and index $n$ (since the index is the smallest dimension of a linear 
  subvariety of the Brauer--Severi variety and such varieties are 
  preserved by specialisation). In particular, it is in the image of 
  the inclusion map $H^1(M, \mu_{m}) \to H^1(M,\mu_{n})$.

  Let $E'$ be the quotient of $E$ by $\mu_m$, so $E'$ is also an elliptic 
  curve over $M$. Let $I_n \subset 
  E'[n]$ be the image of $E[n]$, so we have  exact sequences 
  \[ 
  1 \to \mu_m \to E[n] \to I_n \to 0 \ \mbox{ and } 
  1 \to \mu_{n/m} \to I_n \to \Z/n \to 0 \ . 
  \] 
  By construction, the boundary map of the second sequence maps the 
  element $(t) \in H^1(M,\Z/n)$ to $0$ in 
  $H^1(M,\mu_{n/m})$, hence $(t)$ lifts to an element $\gamma 
  \in H^1(M,I_n)$. Clearly $\gamma$ is mapped to $\beta \in 
  H^2(M,\mu_m)$ by the boundary map of the first exact sequence.

  Now let $M' = M(t^{1/n}) = K'(t^{1/n})((q))$. The 
  restriction of $\gamma$ in $H^1(M', I_n)$ goes to $0$  
  in  $H^1(M', \Z / n \Z)$ by 
  construction, hence it comes from an element of $H^1(M', 
  \mu_{n/m})$.  We have a commutative diagram 
  \[ 
  \xymatrix{ 
    H^1(M',\mu_{n}) \ar[r] \ar[d] &  H^1(M', E) \ar[d] \\ 
    H^1(M',\mu_{n/m}) \ar[r] & H^1(M',E') 
} 
\] 
where the vertical maps are induced by quotienting by $\mu_m$.  The 
first vertical map is surjective and the inclusion $\mu_{n} \to 
E(\ol{M'})$ factors as $\mu_{n} \to \ol{M'}^{\times} \to 
E(\ol{M'})$, so it follows from Hilbert's Theorem 90 that the bottom 
horizontal map is zero. Therefore $\theta$, the image of 
$\gamma$ in $H^1(M,E')$, restricts to $0$ in $H^1(M',E')$. 
 
Let $P$ be the $E'$-torsor corresponding to $\theta$, so $\Pic^0_P$ is 
canonically isomorphic to $E'$. The image of $E[m]$ in $E'$ is 
naturally isomorphic to $\Z/m\Z$; let $\sigma$ denote $1 \in \Z/ m \Z 
\subset E'(M) = \Pic_P(M)$. Pushing out the exact 
sequence 
\[ 
1 \to \mu_m \to E \to E' \to 0 
\] 
via the inclusion $\mu_m \to \G_m$ we get an exact sequence 
\[ 
1 \to \G_m \to S \to E' \to 0 
\] 
whose class in $Ext^1(E', \G_m)$ generates the kernel of the map 
$Ext^1(E', \G_m) \to Ext^1(E, \G_m)$. Under the canonical isomorphisms 
$Ext^1(E', \G_m) \cong \Pic^0_{E'}(M) \cong E'(M)$, $1 \in \Z/m \Z 
\subset E'(M)$ is a generator of the above kernel, so it follows that the 
two elements generate the same subgroup of $Ext^1(E', \G_m)$. 
 
It now follows from Lemma \ref{lem:ext} that $\delta(\sigma)$ and $\beta$ 
generate the same subgroup of $\Br(M)$; in particular, 
$\delta(\sigma)$ has index $n$.  Since $\theta$ becomes $0$ in 
$H^1(M',E')$, it follows that $P_{M'} \cong E'_{M'}$. Since $E'(M)$ is 
infinite, so is $E(M')$ and therefore also $P(M')$ 
 
We conclude that $M$, $P$, $\sigma$ and $M'$ satisfy all the 
conditions of the proposition. 
\end{proof}

\begin{proof}[Proof of Theorem \ref{thm:torsion}] 
  If $g=1$ the result follows from Proposition \ref{prop:ell} so we may 
  assume that $g>1$. 
 
 
  Let $r = g-1/n$ and let $M$, $P$, $\sigma$ and $M'$ be as in 
  Proposition \ref{prop:ell}.  Note that since the index of 
  $\delta(\sigma)$ is $n$, any closed point of $P$ must have degree 
  divisible by $n$.  Let $p_1,p_2,\dots, p_r$ be distinct closed 
  points of $P$ of degree $n$ and let $Y$ be the stable curve over $M$ 
  obtained by gluing two copies of $P$ along all the $p_i$'s, i.e. the 
  $p_i$ in one copy is identified with the $p_i$ in the other copy 
  using the identity map on residue fields. The arithmetic genus of 
  $Y$ is $ 1 + 1 + rn - 1 = 1 + rn = g$.  Using the natural map $\pi: 
  Y \to P$ which is the identity on both components, we get a morphism 
  $\pi^*: \Pic_P \to \Pic_Y$ and we let $\sigma' = \pi^*(\sigma) \in 
  \Pic_Y(M)$.  Note that $\delta(\sigma) = \delta(\sigma') \in 
  \Br(M)$. 
 
  Let $R = M[[x]]$ and let $f:\mf{Y} \to \Spec \ R$ be a generic 
  smoothing of $Y$. So $\mf{Y}$ is a regular scheme and $f$ is a flat 
  proper morphism with closed fibre equal to $Y$ (see for example 
  \cite[Section 1]{deligne-mumford}). By a theorem of Raynaud 
  \cite[Th\'eor\`eme 8.2.1]{raynaud-picard}, $\Pic^0_{\mf{Y}/R}$ is 
  representable by a separated and smooth group scheme of finite type 
  over $R$.  Since $\chr(M) \nmid m$, the endomorphism of 
  $\Pic^0_{\mf{Y}/R}$ given by multiplication by $m$ is etale. Since 
  $R$ is complete, it follows that $\sigma'$ can be lifted to an 
  element $\boldsymbol{\sigma}$ in $\Pic^0_{\mf{Y}/R}(R)$ of order 
  $m$. 
   
  Consider $\delta(\boldsymbol{\sigma}) \in \Br(R)$. Since $\Br(R) = 
  \Br(M)$, we see by the functoriality of the exact sequences in 
  (\ref{seq:relpic}) that $\delta(\boldsymbol{\sigma}) = \delta(\sigma')= 
  \delta(\sigma)$. 
   
  Now let $L = M((x))$, let $X$ be the generic fibre of $f$ and let 
  $\tau$ be the restriction of $\boldsymbol{\sigma}$ in $\Pic^0_X(L)$; 
  by the genericity of the deformation it follows that 
  $Aut(X_{\ol{L}}) = \{Id\}$ if $g>2$. Again by the functoriality of 
  the exact sequences in \eqref{seq:relpic} we see that $\delta(\tau)$ 
  is the image of $\delta(\boldsymbol{\sigma}) = \delta(\sigma)$ in 
  $\Br(L)$.  Thus $\delta(\tau)$ has index $n$ as required. 
\end{proof} 
 
Theorem \ref{thm:double} is a simple consequence of Theorem 
\ref{thm:torsion}. 
\begin{proof}[Proof of Theorem \ref{thm:double}] 
  Since $\mc{R}_g$ is a smooth irreducible Deligne--Mumford stack of 
  dimension $3g -3$, it follows from Theorem~\ref{thm:generic} that to 
  compute $\ed \mc{R}_g$ when $\mc{R}_g$ is tame it suffices to 
  compute the index of the generic gerbe. 
 
  The coarse moduli space $\mf{R}_g$ of $\mc{R}_g$ is generically a 
  fine moduli space parametrizing smooth projective curves $X$ of 
  genus $g$ over $S$ with a non-trivial element of order $2$ of 
  $\Pic_{X/S}(S)$.  Thus over the generic point $\Spec \ K(\mf{R}_g) 
  \in \mf{R}_g$ we have a smooth projective curve $C$ of genus $g$ and 
  an element $\sigma\in \Pic_C(K(\mf{R}_g))$ of order $2$.  It follows 
  that the element of $\Br(K(\mf{R}_g))$ represented by the generic 
  gerbe of $\mc{R}_g$ is the obstruction to the existence of a line bundle 
  $\mc{L}$ over $C$ whose class in $\Pic_C(K(\mf{R}_g))$ is equal to 
  $\sigma$. 
 
  If $b=0$, then $g-1$ is odd hence the generic gerbe is trivial. So 
  assume $b>0$ and let $X$, $L$ and $\sigma $ be obtained by applying 
  Theorem \ref{thm:torsion} with $m=2$ and $n=2^b$.  Since 
  $Aut(X_{\ol{L}}) = \{id\}$ it follows that the image of the map 
  $\Spec \ L \to \mf{R}_g$ lies in the smooth locus $\mf{R}_g'$ of 
  $\mc{R}_g$. Since the restriction of the map $\mc{R}_g \to \mf{R}_g$ 
  is a $\mu_2$ gerbe, it follows that the index of the generic gerbe 
  is $\geq 2^b$. Since the index must also divide $g-1$ it follows 
  that we must have equality as claimed. 
 
\end{proof} 
 
\subsection{The essential dimension of $\mc{A}_1$ over arbitrary fields} 
\label{sec:a1} 
 
We do not know the essential dimension of $\mc{A}_g$ over fields 
of small characteristic. However, 
it follows from classical 
formulae~\cite[Appendix A, Proposition 1.1]{Silverman} 
that $\ed \mc{A}_1 = 2$ over any field of 
characteristic $\neq 2$ and $\ed \mc{A}_1 \leq 3$ over any field of 
characteristic $2$. We prove here the following 
\begin{thm} 
\label{thm:eda1} 
$\ed \mc{A}_1 = 2$ over any field of characteristic $2$. 
\end{thm} 
 
\begin{proof} 
It suffices to prove the theorem over $\F_2$ since it is easy to see 
that $\ed \mc{A}_1 \geq 2$ over any field.  
 
Any elliptic curve $E$ over a field $K$ of characteristic $2$ with 
$j(E) \neq 0$ has an affine equation \cite[Appendix A]{Silverman} 
\[ 
y^2 + xy = x^3 + a_2x^2 + a_6, \ \ \ a_2, 0 \neq a_6 \in K \ , 
\] 
hence it suffices to compute the essential dimension of the residual 
gerbe corresponding to elliptic curves $E$ with $j(E) = 0$. Any such 
curve has an affine equation 
\[ 
y^2 + a_3y = x^3 + a_4x + a_6, \ \ a_3 \neq 0,a_4,a_6 \in K \ . 
\] 
We let $E$ be the curve corresponding to the equation $y^2 + y = x^3$ 
over $\F_2$ and denote by $\Aut(E)$ its automorphism group scheme. 
 
By \cite[Appendix A, Proposition 1.2]{Silverman} and its proof, 
$\Aut(E)$ is an etale group scheme over $\F_2$ of order $24$. As a 
scheme it is given by the equations $U^3 = 1$, $S^4 + S = 0$ and $T^2 
+ T = 0$, where $U,S,T$ are coordinates on $\A^3$. Given a solution 
$(u,s,t)$ of these equations, the corresponding automorphism $E \to E$ 
is given in the above coordinates by $(x,y) \mapsto (x',y')$ with $x = 
u^2x' + s^2$ and $y = y' + u^2sx' + t$. Thus, if $f_i:E \to E$, 
$i=1,2$, over a field $K$ is given by a tuple $(u_i,s_i,t_i)$ then 
$f_2 \circ f_1:E \to E$ is given by the coordinate change 
\[ 
x = u_1^2x_1  + s_1^2 = u_1^2(u_2^2x_2 + s_2^2) + s_1^2 = (u_1u_2)^2x_2 + (u_1s_2 + s_1)^2 
\] 
and  
\begin{multline*} 
y = y_1 + u_1^2s_1x_1 + t_1 = (y_2 + u_2^2s_2x_2 + t_2) + u_1^2s_1(u_2^2x_2 + s_2^2) + t_1 \\ 
= y_2 + (u_1u_2)^2(u_1s_2 + s_1)x_2 + (t_1 + u_1^2s_1s_2^2 + t_2) \ . 
\end{multline*} 
Thus $f_2 \circ f_1$ corresponds to the triple $(u_1u_2, u_1s_2 + 
s_1,t_1 + t_2 + u_1^2s_1s_2^2 )$. 
 
Clearly $\Aut(E)$ becomes a constant group scheme over any field 
containing $\F_4$; one may see that this constant group scheme is 
isomorphic to $SL_2(\F_3)$ by considering its action on $E[3]$. The 
centre of $\Aut(E)$ is the constant group scheme $\Z/2$, the 
non-trivial element corresponds to the tuple $(1,0,1)$ and acts by 
multiplication by $-1$ on $E$. Let $G$ be the quotient of $\Aut(E)$ by 
its centre. It is given by the equations $U^3 = 1, S^4 + S = 0$ and 
the quotient map corresponds to forgetting the last coordinate. 
 
Let $B \subset SL_2(\F_4)$ be the subgroup of upper triangular 
matrices, viewed as a closed subgroup scheme of $SL_{2,\F_2}$ in the 
natural way.  The formula for compostion in $\Aut(E)$ given above 
shows that the map on points $G \to B$ given by $(u,s) \mapsto \bigl 
( \begin{smallmatrix} u & us \\ 0 & u^2 \end{smallmatrix} \bigr )$ 
induces an isomorphism of group schemes over 
$\F_2$.  Thus $G$ is a closed subgroup scheme of $GL_{2,\F_2}$ which 
maps injectively into $PGL_{2,\F_2}$, so $\ed G = 1$. 
 
Now we have a central extension of group schemes over $\F_2$, 
\[ 
0 \to \Z/2 \to \Aut(E) \to G \to 1 \ , 
\] 
which for any extension field $K$ of $\F_2$ gives rise to an exact 
sequence of pointed sets 
\[ 
H^1(K,\Z/2) \sr{\alpha}{\to} H^1(K,\Aut(E)) \sr{\beta}{\to} H^1(K,G) 
\sr{\partial}\to H^2(K,\Z/2) . 
\] 
Since $H^2(K,\Z/2) = 0$ it follows that $\beta$ is surjective. Thus 
$H^1(K,\Z/2)$ operates on $H^1(K,\Aut(E))$ and the quotient is 
$H^1(K,G)$ by \cite[III, Proposition 3.4.5 (iv)]{G}. Since both 
$\Z/2$ and $G$ have essential dimension $1$, it follows that $\ed 
\Aut(E) \leq 2$. 
 
The residual gerbe at the point $E$ of $\mc{A}_1$ is neutral, so it is 
isomorphic to $\mc{B} \Aut(E)$, hence has $\ed \leq 2$. Since the 
generic gerbe is isomorphic to $\mc{B}\ \Z/2\Z$, we conclude that $\ed 
\mc{A}_1 = 2$. 
\end{proof} 
 
 
\bigskip 
 
\noindent {\bf Acknowledgements.} 
I 
thank Arvind Nair and Madhav Nori for 
useful conversations. 
 
 
 
 
\bibliographystyle{amsalpha}
\bibliography{ed}
\end{document}